\documentclass[12pt,reqno]{amsart}
\usepackage{graphics,amsmath,amsfonts,amssymb}%,showkeys} 
\usepackage{epsfig,verbatim}
\usepackage[colorlinks,linkcolor=blue,
citecolor=red,linktocpage=true]{hyperref}
\usepackage{color,xcolor}

 \graphicspath{{./figs/}}

\usepackage{tikz}

\newtheorem{theorem}{Theorem}[section]

\DeclareMathOperator{\csch}{csch}
\def\sympow{{\setbox0\hbox{$\bigcirc$}\setbox1\hbox to\wd0{\hss$s$\hss}%
\wd1 0pt\box1\box0}}%symmetric power

\def\sech{\, {\rm sech}\,} 
\begin{document}
\title[Restricted $N$-body problems]{Corrigendum to: Oscillatory motions in restricted $N$-body problems [J. Differential Equations 265 (2018) 779--803] }
\author[Alvarez, Garc\'{\i}a, Palaci\'an, Yanguas ]{M. Alvarez-Ram\'{\i}rez*, A. Garc\'{\i}a*, J.F. Palaci\'an**, P. Yanguas**}
\address{*Dpto. Matem\'aticas, UAM--Iztapalapa, San Rafael Atlixco 186, Col. Vicentina, 09340 Iztapalapa, M\'exico City, M\'exico.}
\address{**Dpto. Estad{\'\i}stica, Inform\'atica y Matem\'aticas and Institute for advanced materials (InaMat), Universidad P\'ublica de Navarra, 31006 Pamplona, Spain.}

\email{mar@xanum.uam.mx, agar@xanum.uam.mx}

\email{palacian@unavarra.es, yanguas@unavarra.es}
\date{}
\keywords{Restricted $N$-body problem; central configurations; cometary case; symplectic scaling; invariant manifolds at infinity; McGehee's coordinates; Melnikov function; transversality of manifolds; Smale's horseshoe; oscillatory motions}
\maketitle

\begin{abstract}
We consider the planar restricted $N$-body problem where the $N-1$ primaries are assumed to be in a central configuration whereas the infinitesimal particle escapes to infinity in a parabolic orbit. We prove the existence of transversal intersections between the stable and unstable manifolds of the parabolic orbits at infinity which guarantee the existence of a Smale's horseshoe. This implies the occurrence of chaotic motions but also of oscillatory motions, that is, orbits for which the massless particle leaves every bounded region but it returns infinitely often to some fixed bounded region. Our achievement is based in an adequate scaling of the variables which allows us to write the Hamiltonian function as the Hamiltonian of the Kepler problem multiplied by a small quantity plus higher-order terms that depend on the chosen configuration. We compute the Melnikov function related to the first two non-null perturbative terms and characterize the cases where it has simple zeroes. Concretely, for some combinations of the configuration parameters, i.e. mass values and positions of the primaries, the Melnikov function vanishes, otherwise it has simple zeroes and the transversality condition is satisfied. The theory is illustrated for various cases of restricted $N$-body problems, including the circular restricted three-body problem. No restrictions on the mass parameters are assumed.
\end{abstract}

\section{Introduction}
In restricted $N$-body problems there exists a class of unbounded orbits called {\em oscillatory}, where the motion of the $N-1$ primaries remain bounded while the motion of the infinitesimal mass  is unbounded, but nevertheless it leads back near the primaries an infinite number of times. This type of motion was hypothesized by Chazy \cite{chazy} for the 3-body problem, Sitnikov \cite{sitnikov} was the first to construct an  oscillatory motion in a restricted 3-body problem, he also described an initial condition set for such solutions. Later on, using the theory of quasi-random dynamical systems, Alekseev \cite{alekseev} studied the Sitnikov problem showing the  existence of oscillatory motion for  small but positive infinitesimal mass. By the same time Melnikov \cite{melnikov63} and Arnold \cite{arnold64} developed a method to study the formation of transversal intersections of stable and unstable manifolds. This technique, now known as the Melnikov method, replaces variational equations by the computation of certain integrals. 

An orbit  of an infinitesimal particle in a restricted $N$-body problem is {\em parabolic} if the infinitesimal particle escapes to infinity with zero limit radial velocity. In \cite{mcgehee} McGehee  introduced a suitable set of coordinates  that brings the infinity into the origin. He also proved that the set of parabolic orbits is formed by two smooth  manifolds, that can be regarded as the stable and unstable manifolds of an unstable periodic orbit, or a hyperbolic fixed point  in a suitable Poincar\'e map at infinity. This result was used by Moser \cite{moser} to clarify the proof of the existence of oscillatory motion in the Sitnikov problem. The key mechanism  is to show the existence of transversal homoclinic intersections of the parabolic manifolds, which leads to a Smale's horseshoe map that guarantees the existence of symbolic dynamics, giving rise to the existence of oscillatory orbits as a consequence.   

Llibre and Sim\'o \cite{LlibreSimo} followed Moser's approach to prove the existence of oscillatory solutions in the planar circular restricted 3-body problem (RPC3BP). They achieved it by demonstrating  the transversal intersection of the stable and unstable parabolic manifolds for a large Jacobi constant $C$ and a sufficiently small mass ratio $\mu$ between the primary bodies. Xia \cite{Xia92} treated the RPC3BP by the Melnikov method, where $\mu$ is used again as a perturbation parameter. He proved the transversality of the homoclinic manifolds for sufficiently small $\mu$, and $C$ close to $\pm\sqrt{2}$. After that, he used analytic continuation to extend the transversality to almost any value of $\mu$ with $C$ large enough. 

Guardia et al. \cite{famous} demonstrated  that oscillatory motions do occur in the RPC3BP for any $\mu\in (0,1/2]$ and a sufficiently large Jacobi constant. Their result was achieved through an asymptotic formula of the distance between the stable and unstable manifolds of infinity in a level set of the Jacobi constant, and making $C$ sufficiently large, they proved that these manifolds intersect transversally. More recently, Guardia et al. \cite{GMSS2017} followed the above method  to prove the existence of oscillatory solutions in the planar elliptic restricted 3-body problem, for any mass ratio $\mu\in (0,1/2]$ and small eccentricities of the Keplerian ellipses (the primaries perform nearly circular orbits). This is done by constructing an infinite transition chain of fixed points of the Poincar\'e map and then providing a lambda lemma which gives the existence of an orbit which shadows the chain.  

As far as the authors are aware, there are few studies in the literature related with oscillatory motions in restricted $N$-body problems for $N$ greater than or equal to four, see for instance \cite{chinos}. 

Our goal is to analyze the planar restricted $N$-body problem where the $N-1$ primaries form a given central configuration and the infinitesimal particle goes to infinity in a parabolic orbit. More specifically, we show the existence of the stable and unstable manifolds of the parabolic orbits at infinity by means of a Poincar\'e map and establish the transversal intersections between them. This is done by the application of McGehee's Theorem for the existence of manifolds associated to degenerate fixed points \cite{mcgehee}. The transversal intersections allow us to prove the existence of a Smale's horseshoe \cite{gh,Robin} and the subsequent occurrence of chaotic motions. Besides the appearance of oscillatory motions is ensured. These motions correspond to orbits such that the infinitesimal particle leaves every bounded region but it returns infinitely often to some fixed bounded region \cite{moser}. 

The approach carried out here does not require any restriction in the mass parameters of the primaries. Instead of using a small parameter related to the masses we make a symplectic scaling of the variables so that the infinitesimal particle is placed far away from the center of mass of the primaries. This is the so called {\em cometary problem} \cite{meyer81,meyer99}. The scaling is performed through the introduction of a small parameter that measures the distance from the infinitesimal particle to the origin. Thence, the resulting system is given by the Kepler Hamiltonian multiplied by a power of the small parameter plus higher-order terms that are easily obtained through Legendre polynomials. The perturbation depends on the specific configuration of the primaries.

The transversal intersection between the manifolds is proved by applying the Melnikov's method \cite{HolMar1,HolMar2,gh,Robin} related to the first two non-null perturbative term. Indeed we compute a general Melnikov function that works for all the configurations, determining when it has simple zeroes. There are some cases where the calculation of higher-order terms of the Melnikov's function are required. This concerns to two parameters which are given in terms of the masses and positions of the primaries.

In the context of oscillatory motions we generalize the analysis performed for the RPC3BP by considering any central configuration of the $N-1$ primaries. A key point in our analysis is that, instead of using a small parameter related to the masses we apply a symplectic scaling. This allows us to treat all the problems together so that we can check whether the Melnikov function for a specific configuration has simple zeros after replacing the coefficients of the configuration. In this manner we have been able to deal with a large variety of restricted problems in a systematic and straightforward way. 

This paper has been structured as follows. In Section \ref{sec1} we formulate the Hamiltonian of the restricted $N$-body problem where the primaries are in central configuration in a rotating and an inertial frame. We also define the cometary case, introducing an appropriate small parameter so that the problem is expressed as a 2-body Kepler Hamiltonian multiplied by a power of the small parameter plus a small perturbation. In Section \ref{sec2}, we introduce McGehee coordinates to study the behavior of the system in a vicinity of infinity as a Poincar\'e map near a homoclinic orbit of a degenerate periodic orbit with analytic stable and unstable manifolds. Section \ref{sec3} is devoted to the unperturbed problem that can be expressed as a Duffing oscillator. The main results will be given in Section \ref{sec4}, where we establish the existence of transversal homoclinic intersections of stable and unstable manifolds for the perturbed Hamiltonian system. We analyze the Melnikov function up to perturbation orders 4 and 6, depending on the non vanishing order of the Melnikov function. This is the fundamental part in our study, since it gives a systematic way of concluding the transversality between the corresponding invariant manifolds which only depends on the mass parameters and the configuration of the primaries. Then, we use that the Melnikov function has simple zeroes to show the transverse homoclinic intersections in the perturbed problem. In general, the corresponding integrals are hard to calculate, however, we make use of the asymptotic estimates provided in \cite{LlibreSimo,Regina,RegSim}. In Section \ref{sec5}  we illustrate the theory developed in the previous sections in some restricted $N$-body problems. In the RPC3BP, we show that in general it is enough to calculate the terms of order 4 and 6 in $\varepsilon$. In addition to the above, other examples are considered, such as the restricted 4-body problem with primaries in equilateral (Lagrange) configuration, the 5-body problem with primaries in rhomboidal configuration, the collinear restricted $N$-body problem and some polygonal restricted $N$-body problems. Finally, the study of the Melnikov functions is addressed in the Appendix.

All the numeric and symbolic calculations have been performed with Mathematica. We have made the computations within 50 significant digits although we only show the first eight.

\section{Problem statement}\label{sec1}
The planar restricted $N$-body problem is the study of the motion of an infinitesimal mass particle subject to the Newtonian gravitational attraction of $N-1$ bodies called the primaries. It is assumed that the masses of the primaries are so big in relation to the mass of the infinitesimal particle that the latter does not exert any significant influence in the primaries, hence the motion of the primaries becomes an ($N-1$)-body problem. Along this paper the primaries move in a central configuration rotating around their center of mass with a constant angular velocity. Without loss of generality we also set the total mass of the primaries to one and their angular velocity $\omega=1$. 

The Hamiltonian that governs the motion of the infinitesimal particle in  a rotating frame is  
\begin{equation}\label{pico}
H(Q,P)= \frac{1}{2} \left| P \right|^{2} - Q^TJ P - U(Q), 
\end{equation}
where $Q, P\in \mathbb{R}^2$ are the position and momentum, $J=\left(\begin{smallmatrix} 0&1\\-1&0 \end{smallmatrix} \right)$ and the potential is given by
\begin{equation}
U(Q ) = \sum_{k=1}^{N-1} \frac{m_k}{\left|Q-a_k \right|}.
\end{equation}
(We will at times identify $\mathbb{R}^2$ and $\mathbb{C}$, i.e., $(x, y) \in \mathbb{R}^2$ with $x+y\,i \in \mathbb{C}$.)The terms $a_k=a_{k1}+ a_{k2}\,i$ and $m_k$ correspond to the position and mass respectively, of the $k$-th primary, $k=1,\dots, N-1$. Since they are in central configuration the following equations are satisfied:  
\begin{equation}\label{cc}
	a_k= -
	\sum_{\stackrel{j=1}{j\neq k}}^{N-1} 
	\frac{m_j (a_j - a_k)}{|a_j-a_k|^3}, 
	\;\; k=1,\dots , N-1\; \mbox{ and }
	\; \sum_{k=1}^{N-1} m_k a_k=0,
\end{equation} 
with $m_1+ \cdots + m_{N-1}=1$. We observe that the Hamiltonian $H$ represents an autonomous system with two degrees of freedom. This quantity is preserved through the changes of coordinates in spite of the fact that some characteristics, like the time independence of the flow or the Hamiltonian structure, are lost. More details related to the restricted $N$-body problem can be seen in \cite{meyer81,meyer99}.

In order to work in an inertial frame we define the change of coordinates $(Q,P) = e^{-t J}(q,p)$. The Hamiltonian accounting for the motion of the infinitesimal particle in inertial coordinates yields
\begin{equation}
H(q,p,t) = \frac{1}{2} \left| p \right|^{2} - U(q,t), \label{pez}
\end{equation}
where the potential is given by
\begin{equation}
U(q,t) = \sum_{k=1}^{N-1} \frac{m_k}{\left|e^{-t J}q-a_k \right|}.\end{equation} 
{Now the Hamiltonian is time dependent.} 

We are interested in the study of the motion of the infinitesimal particle near parabolic orbits, that is when the particle escapes to infinity with zero limit velocity. Thus, it is convenient to scale the Hamiltonian by introducing a small positive parameter $\varepsilon$ through the change $q \to \varepsilon^{-2}q$, $p\to \varepsilon p$. This is a symplectic transformation with multiplier $\varepsilon$. Hence $H \to \varepsilon H$. This problem now is the cometary regime of the restricted $N$-body problem, see \cite{meyer81,meyer99,meyer09}.

Thus, Hamiltonian (\ref{pez}) becomes
\begin{equation}\label{h1}
	H_\varepsilon(q,p,t) = \varepsilon^3 \left(\frac{1}{2} |p|^2 -  
	\sum_{k=1}^{N-1}  
	\frac{m_k}{\left|\varepsilon^{-2}e^{-t J}q-a_k \right|} \right).
\end{equation}
The expansion of the potential in (\ref{h1}) in terms of Legendre polynomials yields 	
\begin{equation}\label{hamm}
	H_\varepsilon(q,p,t) =   	
	\varepsilon^3 \left( \frac{1}{2}  |p|^2  
	- \frac{1}{|q|} \right)
	- \sum_{j=2}^\infty \frac{\varepsilon^{2j+3}}{|q|^{j+1}}
	\sum_{k=1}^{N-1} m_k |a_k|^j 
	P_j(\cos\gamma_k),
\end{equation}	 
where $P_j$ is the $j$-th term of the Legendre polynomial, and $\gamma_k$ is the angle between the $k$-th primary's position and $e^{-t J}q$.
Let us note that the zero term of the sum is $-1/|q|$ and the next term is zero because we have placed the center of mass at the origin. Hence the problem is a Kepler problem (multiplied by $\varepsilon^3$) plus a perturbation term of order $\varepsilon^{7}$. The series is convergent in the region of the configuration space where $\varepsilon^2 |a_k| < |q|$.

Now we make the symplectic polar change of variables 
\[ q = r e^{i \theta}, \qquad p = R e^{i \theta} + \frac{\Theta}{r} i e^{i \theta}, \]
where $r$ is the distance of the infinitesimal particle to the origin, $\theta$ is the argument of latitude, $R$ stands for the radial velocity and $\Theta$ for the angular momentum. For our analysis we restrict $\Theta$ to be non-zero and bounded, which is equivalent to consider the angular momentum before the scaling being a big quantity as long as $\varepsilon$ is small.

Let  $\alpha_k$ be the angle between the position of $k$-th primary and the horizontal axis of the inertial frame, thus $ \gamma_k = \alpha_k - (\theta- t)$. Then, Hamiltonian (\ref{hamm}) can be written as
\begin{equation} \label{hamm3}
\hspace*{-0.2cm}\begin{array}{rcl}
	H_\varepsilon(r, \theta, R, \Theta,t) &\hspace*{-0.2cm}=\hspace*{-0.2cm}&
	\displaystyle \varepsilon^3 \left(\frac{1}{2} \left( R^2 +\frac{\Theta^2}{r^2} \right) - \frac{1}{r}  \right) \\[1ex]
	&& \hspace*{-0.19cm}\displaystyle - \sum_{j=2}^{\infty}  \frac{\varepsilon^{2j +3}}{r^{j+1}} \sum_{k=1}^{N-1}  m_k  
	|a_k|^j  P_j \left(\cos  \left( \alpha_k  -(\theta- t) \right)  \right). 
	\end{array} \end{equation}
We observe that the argument of the Legendre polynomial $P_j$ depends on $a_k$ through the relations $\cos (\alpha_k)=a_{k1}/|a_k|$ and $\sin (\alpha_k)=a_{k2}/|a_k|$.
	
The associated Hamiltonian system is given by
\begin{eqnarray} \label{pato}
	  \dot r &\hspace*{-0.2cm}=\hspace*{-0.2cm}& \displaystyle \varepsilon^3 R, \nonumber \\	  
	  \dot R &\hspace*{-0.2cm}=\hspace*{-0.2cm}& \displaystyle \varepsilon^3 \left(-\frac{1}{r^2}
	              + \frac{\Theta^2}{r^3} \right) \nonumber \\ && 
	                       \displaystyle -\sum_{j=2}^{\infty} 
				\frac{(j+1) \varepsilon^{2j+3}}{r^{j+2}}
				\sum_{k=1}^{N-1} m_k \left| a_k \right|^j  
				P_j  \left(\cos  \left( \alpha_k - (\theta- t) \right) \right), \\ 
	   \dot \theta &\hspace*{-0.2cm}=\hspace*{-0.2cm}& \displaystyle \varepsilon^3 \frac{\Theta }{r^{2}}, \nonumber \\ 
	   \dot \Theta &\hspace*{-0.2cm}=\hspace*{-0.2cm}& \displaystyle \sum_{j=2}^{\infty}  \frac{ \varepsilon^{2j+3}}{r^{j+1}} 
					\sum_{k=1}^{N-1} m_k  |a_k|^j
					Q_j  \left(\cos \left( \alpha_k  -(\theta- t) \right)  \right) 
					\sin  \left(  \alpha_k  -(\theta- t) \right),  \nonumber
	 \end{eqnarray}
where $Q_j(w) = d P_j(w)/dw$. 

\section{Application of McGehee's Theorem}\label{sec2}
In order to study the motion of the infinitesimal mass near infinity, we make the change of coordinates introduced by McGehee \cite{mcgehee}:  
\[ r= x^{-2}, \qquad R=-\sqrt{2} y, \qquad s = t-\theta, \]
where $s\in \mathbb{S}^1$. In these new coordinates the equations (\ref{pato}) become
\begin{equation}\label{oso}
	\begin{array}{rcl}
	 	\dot{x} &\hspace*{-0.2cm}=\hspace*{-0.2cm}& \mbox{$\frac{1}{\sqrt{2}}$} \varepsilon^3 x^3 y,\\[1.2ex]
	        \dot{y} &\hspace*{-0.2cm}=\hspace*{-0.2cm}& \varepsilon^3 \left( \mbox{$\frac{1}{\sqrt{2}}$} x^4 - \mbox{$\frac{1}{\sqrt{2}}$} \Theta^2\, x^6 \right) \\ && \displaystyle  -\mbox{$\frac{1}{\sqrt{2}}$}
		\sum_{j=2}^{\infty} (j+1) \varepsilon ^{2 j +3} x^{2 j+4 }
	   	\sum_{k=1}^{N-1}  
  	    	m_k \left|a_k\right|^j  
  	    	P_j\left(\cos  \left(\alpha _k+s\right) \right), \\[1.2ex]
	\dot{s} &\hspace*{-0.2cm}=\hspace*{-0.2cm}& 1 - \varepsilon^3 x^4\,\Theta,  \\
	\dot{\Theta} &\hspace*{-0.2cm}=\hspace*{-0.2cm}& \displaystyle \sum_{j=2}^{\infty}   \varepsilon ^{2 j+3} x^{2 j +2}
			\sum_{k=1}^{N-1}  
			m_k \left|a_k\right|^{j}    
			\sin ( \alpha _k+s ) Q_j 
			\left(\cos \left(    \alpha _k+s \right)  \right).
	\end{array}
\end{equation}

Similarly to the circular restricted 3-body problem, the above system can be smoothly extended to $x=0$; in addition, it has a Jacobi-like first integral given by equating Hamiltonian $H$ in (\ref{pico}) to a constant. Let $C$ be an arbitrary value of it. Thus 
\begin{equation}\label{jacobi1} 
C = H_\varepsilon(x,y,\theta, \Theta,t) - \Theta,
\end{equation} 
where $H_\varepsilon(x,y,\theta, \Theta,t)$ is given in (\ref{hamm3}) after replacing $r$ and $R$ for their values in terms of  $x$ and $y$. In order to verify the hypotheses of McGehee's Theorem \cite{mcgehee}, we observe that $\Theta$ can be written in terms of $x$, $y$, $s$  and 
the value of the Jacobi-like integral $C$ as follows:
\begin{equation}\label{gato} 
	 \Theta = \frac{1 \pm \sqrt{1 + 2 \varepsilon^3 x^4 (C+ \varepsilon^3(x^2-y^2))}}{\varepsilon^3 x^4} + O(\varepsilon^{7}).
\end{equation}

We take the negative sign in (\ref{gato}) because we are interested in small values of $x$. Therefore, the variable $\Theta$ depends smoothly on $x\geq 0$, $y\in\mathbb{R}$, $s\in \mathbb{S}^1$. In fact $\Theta =-C$ when $x=y=0$, hence the equation referring to $\dot{\Theta}$ can be dropped in system (\ref{oso}). Expanding (\ref{gato}) in power series around $x=0$, replacing its value in (\ref{oso}), the resulting system becomes:
\begin{equation}\label{puma}
	\begin{split}
	 	\dot{x} &= \mbox{$\frac{1}{\sqrt{2}}$} \varepsilon^3 x^3 y, \\
		\dot{y} &= \mbox{$\frac{1}{\sqrt{2}}$} \varepsilon^3 x^4 
				+\varepsilon^3 x^6 f(x,y,\varepsilon,s,C), \\
		\dot{s} &= 1 + C \varepsilon^3 x^4 
		              + \varepsilon^6 x^6 - \varepsilon^6 x^4 y^2 + 
				\varepsilon^6 x^8 g(x,y,\varepsilon,s,C).  
	\end{split}
\end{equation}

Observe that the functions $f(x,y,\varepsilon,s,C)$ and $g(x,y,\varepsilon,s,C)$ are real analytic.

Now, we are going to exploit the fact that the solution $\gamma (t)$ of system (\ref{puma}) with  $x=0$ and $y=0$ is given by
\begin{equation} 
x(t)=0,  \quad  y(t)=0 , \quad s(t)=s_0 + t\; \mbox{mod}\; 2\pi.
\label{pan}
\end{equation}

This solution does not depend on the Jacobi-like first integral $C$ and is a $2\pi$-periodic solution in $t$. Let $\Sigma= \{ \left(x,y, s  \right): s=s_0 \}$ be a  transversal section to this periodic orbit. This set is parametrized by the coordinates $x$ and $y$. For the point $\kappa_0=\left(0,0,s_0 \right)\in \gamma\cap \Sigma$, the Implicit Function Theorem shows that there exists an open set $V \subset \Sigma$ containing $\kappa_0$ and a smooth function $\sigma: V \to \mathbb{R}$, the return time, such that the trajectories starting in $V$ come back to $\Sigma$ in a time $\sigma$ close to $2\pi$. 

Let $\varphi^t(x_0, y_0) = (x(t, x_0, y_0), y(t, x_0, y_0), s(t, x_0, y_0))$ be the flow solution of (\ref{puma}) with initial condition $(x(0), y(0), s(0))=(x_0,y_0,s_0\,\mbox{mod}\,2\pi)$. This solution depends on $\varepsilon$, but it is usually omitted.

The Poincar\'e map $\mathcal{P}: V\subset \Sigma \to \Sigma$ of the periodic orbit $\gamma$ (or of the fixed point $(0,0,s_0)$) is given by $\mathcal{P}(x,y)=\varphi^{\sigma (x,y)}(x,y)$ where
\begin{equation} \label{poincaremap}
	\mathcal{P}:
	\left\{
		\begin{array}{ll}
			x \to & \hspace*{-0.2cm} x +
				\sqrt{2}\:     \pi   \varepsilon^3 x^3 y + \varepsilon^6 x^7 y \, r_1(x,y,\varepsilon)						  \\[-1ex]
			&\\
			y\to & \hspace*{-0.2cm} y + \sqrt{2}\: \pi   \varepsilon^3 x^4 
				(1-C^2 x^2)+\varepsilon^6 x^6 r_2(x,y,\varepsilon)
		\end{array}
	\right.
\end{equation}
and $r_1$ and $r_2$ are real analytic functions. 

A straightforward computation shows that if we make the transformation $x= u+v$, $y=v-u$, the map $\mathcal{P}$ takes the form 
\[
	\tilde{\mathcal{P}}:
	\left\{
		\begin{array}{ll}
			u \to & \hspace*{-0.2cm} u + \varepsilon^3 p_1(u,v)+  \varepsilon^3 s_1(u,v,\varepsilon) 
									  \\[-1ex]
			&\\
			v\to & \hspace*{-0.2cm} v + \varepsilon^3 p_2(u,v) + \varepsilon^3 s_2(u,v,\varepsilon)
		\end{array}
	\right.
\] 
where $p_1(u,v)=-\sqrt{2} \pi u (u+v)^3$, $p_2(u,v)=\sqrt{2} \pi v (u+v)^3$ and $s_1$ and $s_2$ are real analytic functions starting with homogeneous polynomials of degree 6 in $u$ and $v$. For $u>0$ we have $p_1(u,0) = - \sqrt{2} \pi u^4 < 0$, $p_2(u,0)=0$ and $\frac{\partial p_2}{\partial v }(u,0) = \sqrt{2} \pi u^3 >0$.
If $\varepsilon\geq\varepsilon_0>0$ is small enough, these are the hypotheses of McGehee's Theorem \cite{mcgehee}, which we proceed to state in our context now. 

For $\delta=\delta (\varepsilon)>0$ and $\beta=\beta (\varepsilon)>0$, the sector centred on the line $y=-x$ is defined by
$B_1(\delta, \beta )=
\{ (x,y): 0\leq x \leq \delta, 
		(-1-\beta)x\leq y\leq (-1+\beta)x \}$.  
Similarly,  the
sector centred on the line $y=x$ is 
$B_2(\delta, \beta )=
\{ (x,y): 0\leq x \leq \delta, 
		(1-\beta)x\leq y\leq (1+\beta)x \}$.
The sets
\begin{align*}
&W^s_\varepsilon(0,0) =\big\{ 
		(x,y)\in B_1(\delta, \beta ): \;
	\forall n\geq 0, \,\,  \mathcal{P}^n(x,y) \in  B_1( \delta, \beta),  \\
	 & \hspace*{2cm}\; \lim_{n\to \infty}\mathcal{P}^n(x,y) = (0,0) 
\big\},\\[1ex]
& W^u_\varepsilon(0,0) =\big\{ 
		(x,y)\in B_2(\delta, \beta ): \;
	\forall n\leq 0, \,\, \mathcal{P}^n(x,y) \in  B_2( \delta, \beta),  \\
	 &\hspace*{2cm} \; \lim_{n\to -\infty}\mathcal{P}^n(x,y) = (0,0) 
\big\},
\end{align*}
are called the stable and unstable manifolds of the fixed point $(0,0)$.  

\begin{theorem}[McGehee, \cite{mcgehee}]
For the map $\mathcal{P}: V\subset \Sigma \to \Sigma$ given by (\ref{poincaremap}), there is $\delta=\delta(\varepsilon)>0$ and  
$\beta=\beta(\varepsilon)>0$ such that the manifolds $W^s_\varepsilon(0,0)\subset B_1(\delta, \beta )$ and $W^u_\varepsilon(0,0)\subset B_2(\delta, \beta )$ correspond to the graphs 
of two functions $\psi_s$, $\psi_u:[0,\delta] \to \mathbb{R}$ that are smooth, real analytic in $(0,\delta]$, and $\psi_s(0)= \psi_u(0)=0$, $\psi_s'(0)=-1$, $\psi_u'(0)=1$. In addition to that, they vary uniform and smoothly for $\varepsilon\geq\varepsilon_0>0$ small enough. \end{theorem}

The sets $W^s_\varepsilon(0,0)$ and $W^u_\varepsilon(0,0)$ are curves, that is, one-dimensional manifolds. From here it follows that  
\begin{align*}
& W^s_\varepsilon(\gamma)=\left\{
				\varphi^t\left(x_\varepsilon,y_\varepsilon\right): 
					t\geq 0,\; \left(x_\varepsilon,y_\varepsilon\right)\in W^s_\varepsilon(0,0)
			\right\},\\
& W^u_\varepsilon(\gamma)=\left\{
				\varphi^t\left(x_\varepsilon,y_\varepsilon\right): 
					t\leq 0, \; \left(x_\varepsilon,y_\varepsilon\right)\in W^u_\varepsilon(0,0)
			\right\},			
\end{align*}
are smooth manifolds of dimension two. They are formed by the orbits that escape to infinity ($x=0$) with zero velocity ($y=0$). These are called the {\em parabolic manifolds}.

\section{The main term of the Hamiltonian}\label{sec3}
The main part in Hamiltonian (\ref{hamm3}) corresponds to the Kepler problem, and the related equations (\ref{oso}) take the form
\begin{equation} \label{leon} 
\begin{array} {ll}
\dot{x} = \displaystyle \mbox{$\frac{1}{\sqrt{2}}$} \varepsilon^3 x^3 y, & \qquad	
\dot{y} = \displaystyle \varepsilon^3 \left(\mbox{$\frac{1}{\sqrt{2}}$} x^4 - \mbox{$\frac{1}{\sqrt{2}}$} \Theta^2 x^6 \right),\\
& \\
\dot{s} = 1 - \varepsilon^3 x^4 \Theta, 
&\qquad \dot{\Theta} = 0. 
\end{array} 
\end{equation} 
At this point it is convenient to introduce a new time through $d\tau/dt = \varepsilon^3 x^3/\sqrt{2}$.   Fixing $\Theta = \Theta_0$ non zero, the equations for $x$ and $y$ become 
\begin{equation}\label{rata}
	\frac{d\,x}{d\, \tau} = x' = y,\qquad   
	\frac{d\,y}{d\, \tau} = y' = x - \Theta_{0}^2 x^3.
\end{equation}
Thus (\ref{rata}) is a $\Theta_0$-parametrized Duffing equation. The origin $(0,0)$ is a hyperbolic saddle point and its stable and unstable manifolds form a homoclinic orbit  $\xi (\tau)$ given by
\begin{equation}\label{xysol}
x(\tau) = \frac{\sqrt{2}}{|\Theta_0|} \sech {\tau}, \qquad
y (\tau) = -\frac{\sqrt{2}}{|\Theta_0|}\tanh \tau\sech \tau,
\end{equation}
connecting the fixed point with itself as shown in Fig. \ref{fig_duffin}.  
\begin{figure}
\centering
\includegraphics[scale=0.26]{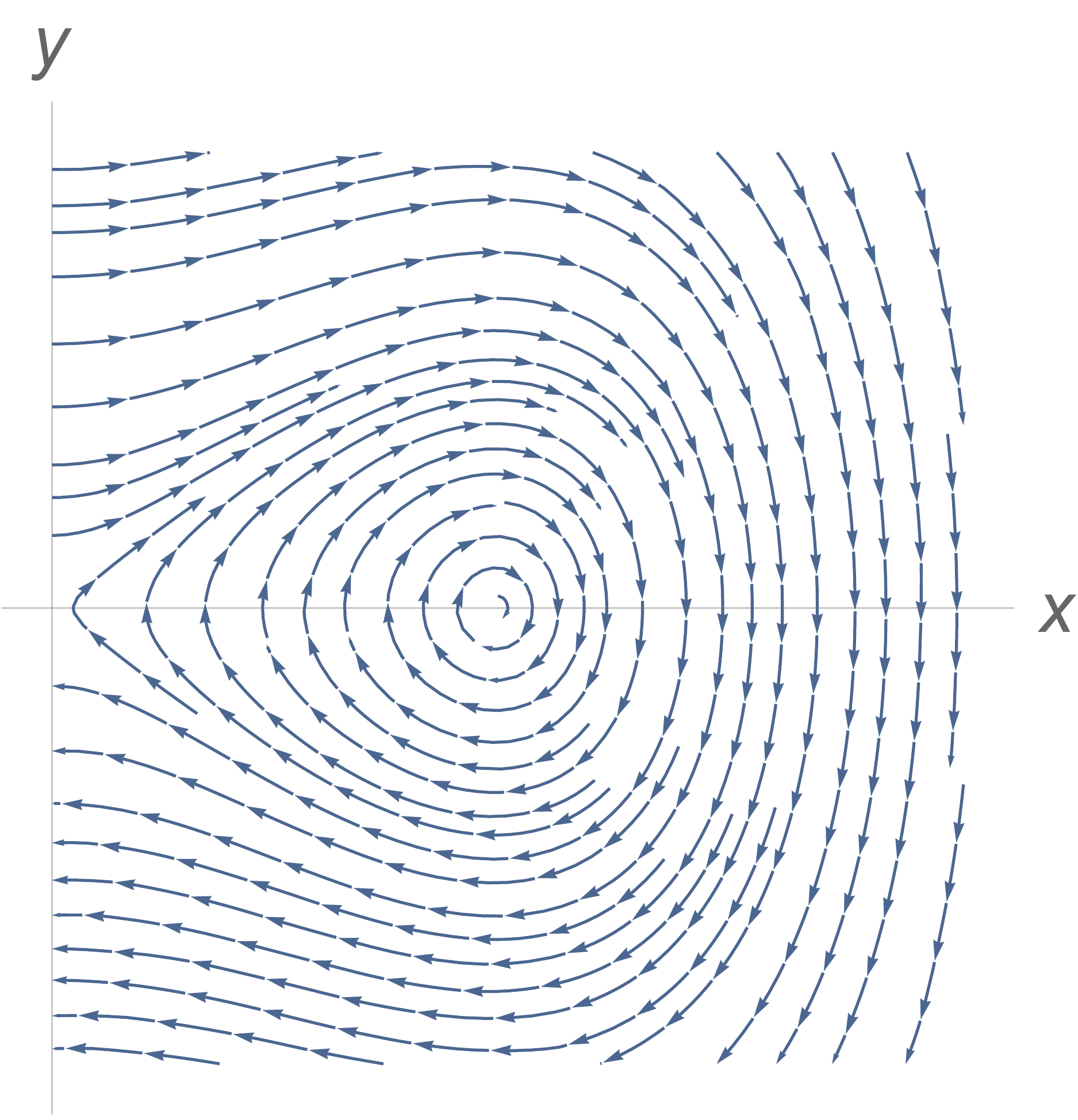} \qquad
\includegraphics[scale=0.26]{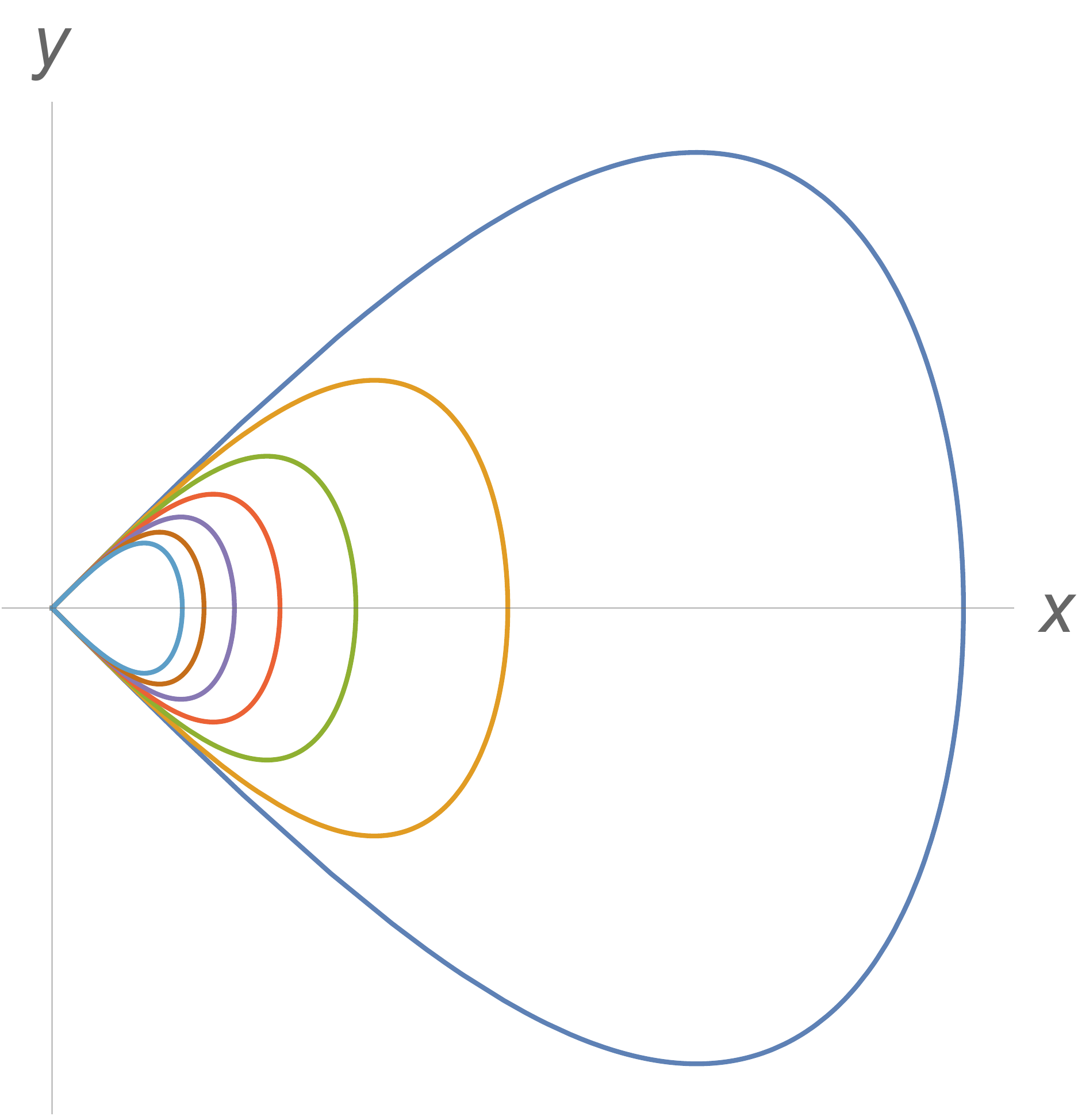}
\caption{From
left to right, we show the flow of the Duffing Hamiltonian and the homoclinic loop for different values of $\Theta_0$.}\label{fig_duffin}
\end{figure}

The Keplerian Hamiltonian associated to (\ref{rata}) is rewritten as the Duffing Hamiltonian:
\begin{equation} \label{ham_e0}
H_D = \mbox{$\frac{1}{2}$} y^2- \mbox{$\frac{1}{2}$} x^2 + \mbox{$\frac{1}{4}$} x^4 \Theta_0^2.
\end{equation}

\section{Perturbed Problem} \label{sec4}

In this section we compare equations (\ref{puma}) and (\ref{leon}). We observe that the Poincar\'e map of the last one has $(0, 0)$ as a fixed point with a homoclinic orbit parametrized by \eqref{xysol}. 
The fixed point is preserved by the Poincar\'e map of \eqref{puma}, however the homoclinic orbit
is broken into the curves $W^s_\varepsilon(0,0)$ and $W^u_\varepsilon(0,0)$. Indeed, the smooth dependence on $\varepsilon$ implies that $W^s_\varepsilon(\gamma)$ and $W^u_\varepsilon(\gamma)$ are parametrized by orbits of the form
\begin{equation}\label{paramet}
\begin{array}{rcl}
	\varphi_s (\tau, x_\varepsilon^s, y_\varepsilon^s)
	= \xi(\tau - \tau_0) 
	+ \varepsilon^4 \varphi_s^1(\tau, \tau_0, \varepsilon) + ..., \text{ for } \tau \geq \tau_0,
\\[1.5ex]
	\varphi_u (\tau, x_\varepsilon^u, y_\varepsilon^u)
	= \xi(\tau - \tau_0) 
	+  \varepsilon^4 \varphi_u^1(\tau, \tau_0, \varepsilon) + ..., \text{ for } \tau \leq \tau_0,
\end{array} \end{equation}
with $\tau_0=\tau(s_0)$, the functions $\varphi_s^1$ and $\varphi_u^1$ are determined by the first variational equation of the $\varepsilon^4$-perturbation of (\ref{rata}) along the orbit $\xi(\tau)$, that is, the corrections to $\xi(\tau - \tau_0)$ corresponding to the perturbation of the Duffing Hamiltonian of order $\varepsilon^4$, the next terms, i.e. $\varphi_s^2$, $\varphi_u^2$, correspond to the perturbation of order $\varepsilon^7$, etc. Besides, $\varphi_s (\tau, x_\varepsilon^s, y_\varepsilon^s)$, $\varphi_u (\tau, x_\varepsilon^u, y_\varepsilon^u)$ are taken so that $\varphi_s (\tau_0, x_\varepsilon^s, y_\varepsilon^s), \varphi_u (\tau_0, x_\varepsilon^u, y_\varepsilon^u)\in \Sigma$, guaranteeing the existence of a parametrization of the manifolds as above, see \cite{Xia92}.

Our purpose now is to determine the speed of breaking up of $W^s_\varepsilon(0,0)$ and $W^u_\varepsilon(0,0)$ (equivalently the breaking up of $W^s_\varepsilon(\gamma)$ and $W^u_\varepsilon(\gamma)$) under the perturbation. We point out that the normal direction to the level sets of $H_D$ is the only one to be taken into account. 

To measure the rate of separation between $W^s_\varepsilon(\gamma)$ and $W^u_\varepsilon(\gamma)$ with respect to $\varepsilon$, we take
\[ H_D( \varphi_s (\tau_0, x_\varepsilon^s, y_\varepsilon^s)) -
   H_D( \varphi_u (\tau_0, x_\varepsilon^u, y_\varepsilon^u)) =
			\varepsilon^\nu {\mathcal M}_\nu(\tau_0,\varepsilon) + ..., \]
where
\begin{align*} 
	{\mathcal M}_\nu(\tau_0,\varepsilon)=& 
	\int_{-\infty}^{\infty} \frac{d H_D}{d \tau} ( \xi (\tau), \tau+\tau_0)\; d\tau
\end{align*}
is the first term of the Melnikov function, $x$, $y$ are evaluated in the unperturbed homoclinic and the derivative of $H_D$ with respect to $\tau$ is considered up to the perturbation of order $\varepsilon^{\nu}$ in (\ref{rata}) such that the function ${\mathcal M}_\nu$ is not identically zero. The ellipsis in the formula measuring the splitting of the manifolds should not be underestimated. More specifically we decompose
\begin{equation}\label{split} 
H_D( \varphi_s (\tau_0, x_\varepsilon^s, y_\varepsilon^s)) -
   H_D( \varphi_u (\tau_0, x_\varepsilon^u, y_\varepsilon^u)) =
			{\mathcal M}(\tau_0, \varepsilon) + {\mathcal R}(\tau_0, \varepsilon), \end{equation}
where \[ {\mathcal M}(\tau_0, \varepsilon) = \sum_{l=2}^\infty \varepsilon^{2l} {\mathcal M}_{2l}(\tau_0,\varepsilon) \] is the so called Melnikov function and ${\mathcal M}_{2l}$ contain the main terms of the total variation of $H_D$; their computation is similar to that of ${\mathcal M}_\nu$ above. In other words, the series contains the whole Hamiltonian $H_\varepsilon$ with $x$, $y$ replaced by their values in (\ref{xysol}) while ${\mathcal R}$ corresponds to the contribution to the separation provided by the higher order terms of the orbits on the stable and unstable manifolds given in (\ref{paramet}). After changing from $\tau_0$ to $s_0$ the infinite series can be interpreted as a Fourier series in $s_0$ whose coefficients are given as asymptotic expressions in terms of $\varepsilon$. In the Appendix we will provide the leading terms of the series as well as the corresponding estimates for the higher orders. Regarding $\mathcal R$ we shall prove that it is smaller than the dominant terms of ${\mathcal M}$ at least in regions of the phase space containing portions of the stable and unstable manifolds of $\gamma$ big enough where transversality can be checked, as it will be detailed in the Appendix.

The consequence of the existence of transversal homoclinic manifolds is the appearance of a Smale's horseshoe, and thence the occurrence of chaotic motion of the infinitesimal particle. This includes the existence of oscillatory motion, see \cite{moser,gh}.
 
In order to apply the results in the previous paragraphs we need to make some arrangements to Hamiltonian (\ref{hamm3}). Up to terms of order $\varepsilon^9$, it reads as
\begin{eqnarray}\label{HP4}
H_\varepsilon(r, \theta, R, \Theta, t) &\hspace*{-0.2cm}=\hspace*{-0.2cm}& \displaystyle \varepsilon^3 H_3 + \varepsilon^7 H_7 + \varepsilon^9 H_9 + O(\varepsilon^{11}) \nonumber\\[1ex]
&\hspace*{-0.2cm}=\hspace*{-0.2cm}& \displaystyle 
\varepsilon^3 \left(\frac{1}{2} \left( R^2+\frac{\Theta^2}{r^2}\right)-\frac{1}{r}\right) \nonumber \\
&&\displaystyle -\frac{\varepsilon^7}{4r^3}\left( c_1 + c_2 \cos 2(t-\theta) + c_3 \sin 2(t-\theta)\right) \\
&&\displaystyle -\frac{\varepsilon^9}{8r^4}\left(  d_1 \cos (t-\theta) + d_2 \sin (t-\theta) + d_3 \cos 3(t-\theta) \right. \nonumber \\[1ex]
&& \displaystyle \left.  \hspace*{1.1cm} + d_4 \sin 3(t-\theta)\right) \nonumber \\[1ex] \displaystyle && +O(\varepsilon^{11}),\nonumber
\end{eqnarray}
where $c_1$, $c_2$ and $c_3$ are the following constant terms, that depend only on $a_k$ and $m_k$:
\begin{equation} \label{c}
\begin{array}{cc} 
c_1 = \displaystyle \sum_{k=1}^{N-1} m_k (a_{k1}^2+a_{k2}^2), \qquad
c_2 = \displaystyle3 \sum_{k=1}^{N-1} m_k (a_{k1}^2-a_{k2}^2), \\[2ex]
c_3 = \displaystyle -6 \sum_{k=1}^{N-1} m_k a_{k1} a_{k2},
\end{array}
\end{equation}
and the $d_j$ are given by
\begin{equation}
\label{ds}
\hspace*{-0.25cm}\begin{array}{ll}
d_1 = \displaystyle 3 \sum_{k=1}^{N-1} m_k a_{k1} (a_{k1}^2+a_{k2}^2), &\,\,
d_2 = \displaystyle -3 \sum_{k=1}^{N-1} m_k a_{k2} (a_{k1}^2+a_{k2}^2), \\[2.5ex]
d_3 = \displaystyle 5 \sum_{k=1}^{N-1} m_k a_{k1} (a_{k1}^2-3a_{k2}^2), &\,\,
d_4 = \displaystyle -5 \sum_{k=1}^{N-1} m_k a_{k2} (3a_{k1}^2-a_{k2}^2).
\end{array}
\end{equation}
Higher order terms appearing in $O(\varepsilon^{11})$ should also be taken into consideration as we shall realize in the proof of the main result of this section.

The next step consists in applying McGehee's transformation to the equations of motion associated to (\ref{HP4}). After replacing $t-\theta$ by $s$, one gets
\begin{equation}
\label{EM4}
\begin{array}{rcl}
\dot{x} &\hspace*{-0.2cm}=\hspace*{-0.2cm}& \mbox{$\frac{1}{\sqrt{2}}$} \varepsilon^3 x^3 y, \\[1.2ex]
\dot{y} &\hspace*{-0.2cm}=\hspace*{-0.2cm}& \mbox{$\frac{1}{\sqrt{2}}$} \varepsilon^3 (1-\Theta^2 x^2) x^4 + \mbox{$\frac{3}{4\sqrt{2}}$} \varepsilon^7 \left(c_1 + c_2 \cos 2 s+ c_3 \sin 2s\right) x^8\\[1.2ex] && + \frac{1}{2 \sqrt{2}} \varepsilon^9 \left( d_1 \cos s + d_2 \sin s + d_3 \cos 3s + d_4 \sin 3s \right) x^{10}\\[1.2ex] && + O(\varepsilon^{11}), \\[1.2ex]
\dot{s} &\hspace*{-0.2cm}=\hspace*{-0.2cm}& 1 - \varepsilon^3 \Theta x^4,  \\[1.2ex]
\dot{\Theta} &\hspace*{-0.2cm}=\hspace*{-0.2cm}& -\mbox{$\frac{1}{2}$} \varepsilon^7\left( c_3 \cos 2 s - c_2 \sin 2s \right) x^6\\[1.2ex] && +\frac{1}{8} \varepsilon^9 \left( d_1 \sin s - d_2 \cos s + 3 d_3 \sin 3s - 3 d_4 \cos 3s \right) x^8\\[1.2ex] &&+O(\varepsilon^{11}),\end{array}\end{equation}
which corresponds to Hamiltonian equations (\ref{oso}).

With respect to the new time $\tau$, Eq. (\ref{EM4}) gets transformed into
\begin{equation}
\label{EM44}
\begin{array}{rcl}
x' &\hspace*{-0.2cm}=\hspace*{-0.2cm}& y, \\[1.2ex]
y' &\hspace*{-0.2cm}=\hspace*{-0.2cm}& (1-\Theta^2 x^2) x + \mbox{$\frac{3}{4}$} \varepsilon^4 \left(c_1 + c_2 \cos 2 s+ c_3 \sin 2s\right) x^5\\[1.2ex]&& +\frac{1}{2} \varepsilon^6 \left( d_1 \cos s + d_2 \sin s + d_3 \cos 3s + d_4 \sin 3s \right) x^7 \\[1.2ex] &&+ O(\varepsilon^8), \\[1.2ex]
s' &\hspace*{-0.2cm}=\hspace*{-0.2cm}& \sqrt{2} (\varepsilon^{-3} - \Theta x^4) x^{-3}, \\[1.2ex]
\Theta' &\hspace*{-0.2cm}=\hspace*{-0.2cm}& -\mbox{$\frac{1}{\sqrt{2}}$}\varepsilon^4 \left( c_3 \cos 2 s - c_2 \sin 2s \right) x^3\\[1.2ex] && +\frac{1}{4 \sqrt{2}} \varepsilon^6 \left( d_1 \sin s - d_2 \cos s + 3 d_3 \sin 3s - 3 d_4 \cos 3s \right) x^5 \\[1.2ex] && + O(\varepsilon^8).
\end{array}\end{equation}
To obtain the Melnikov function we use the corresponding integral of the main term given by $H_D$ in (\ref{ham_e0}), 
and compute the total derivative, assuming that $\Theta_0$ in (\ref{ham_e0}) is considered as $\Theta$, arriving at
\[ \frac{d H_D}{d \tau} = \frac{\partial H_D}{\partial x} x' + \frac{\partial H_D}{\partial y} y' + \frac{\partial H_D}{\partial \Theta} \Theta',
\]
where $x'$, $y'$, $\Theta'$ are given in (\ref{EM44}). In this expression $x$ and $y$ are replaced by their explicit values obtained for the unperturbed problem and that were given in (\ref{xysol}). 

The next step consists in solving the differential equation for $s'$ in (\ref{EM44}) using the fact that $\Theta$ is assumed to be constant, i.e. $\Theta=\Theta_0$, and $x(\tau)$ is taken from the solution (\ref{xysol}). The corresponding equation becomes
\[ s'(\tau) = \pm \frac{1}{2} \varepsilon^{-3} \Theta_0^3 \cosh^3 \tau \mp 2 \sech \tau. \]
We set $s(0) = s_0$  and the solution yields
\begin{equation}
s(\tau) = s_0 \mp 4 \arctan (\tanh( \tau/2 )) \pm \frac{1}{24} \varepsilon^{-3} \Theta_0^3 \left( 9 \sinh \tau + \sinh 3\tau \right).\end{equation}
The upper signs in $s'$, $s$ and the formulae that follow are used when $\Theta_0>0$ and the lower ones when $\Theta_0<0$.

The leading term of the resulting expression of the total derivative of $H_D$ with respect to $\tau$ becomes
\[ \varepsilon^4 M_4 = \mp \frac{3 \varepsilon^4 \sech^{10}\tau \tanh \tau}{4 \Theta_0^6} \left( 8 c_1 \cosh^4 \tau+ A \cos 2s_0+ B \sin 2 s_0\right), \]
where 
\[ \begin{array}{rcl}
A &\hspace*{-0.2cm}=\hspace*{-0.2cm}& (c_2 c_\Theta + c_3 s_\Theta) (35 - 28 \cosh 2 \tau + \cosh 4 \tau) \\[1ex] && - 
8 (c_3 c_\Theta - c_2 s_\Theta) (7 \sinh \tau - \sinh 3 \tau) \\[1ex] &&
-\mbox{$\frac{2}{3}$} \csch \tau \left( 8 (c_2 c_\Theta + c_3 s_\Theta) (7 \sinh \tau - \sinh 3 \tau) \right. \\[1ex] && \hspace*{1.83cm} \left. + (c_3 c_\Theta - c_2 s_\Theta) (35 - 28 \cosh 2 \tau + \cosh 4 \tau) \right), \\[1.5ex]
B &\hspace*{-0.2cm}=\hspace*{-0.2cm}&  (c_3 c_\Theta - c_2 s_\Theta) (35 - 28 \cosh 2 \tau + \cosh 4 \tau) \\[1ex] && + 
8 (c_2 c_\Theta + c_3 s_\Theta) (7 \sinh \tau - \sinh 3 \tau)\\[1ex] &&
- \mbox{$\frac{2}{3}$} \csch \tau \left( (c_2 c_\Theta + c_3 s_\Theta) (35 - 28 \cosh 2 \tau + \cosh 4 \tau)  \right. \\[1ex] && \hspace*{1.83cm} \left.  - 
8 (c_3 c_\Theta - c_2 s_\Theta) (7 \sinh \tau - \sinh 3 \tau) \right), \end{array} \]
and
\[ \begin{array}{c}
c_\Theta = \cos\left(\mbox{$\Theta_0^3 \over 12 \varepsilon^3$} (9 \sinh \tau + \sinh 3 \tau)\right), \\[1.4ex]
s_\Theta = \sin\left(\mbox{$\Theta_0^3 \over 12 \varepsilon^3$} (9 \sinh \tau + \sinh 3 \tau)\right).\end{array} \]

Proceeding as in the previous paragraphs, the total derivative of $H_D$ with respect to $\tau$ corresponding to the terms of order six in $\varepsilon$ is
\[ \begin{array}{rcl}
\varepsilon^6 M_6 &\hspace*{-0.2cm}=\hspace*{-0.2cm}& \displaystyle \mp \frac{\varepsilon^6 \sech^{14}\tau \tanh \tau}{12 \Theta_0^8} \left( C \sin s_0+ D \cos s_0 + E \sin 3 s_0 \right.  \\[0.7ex]  && \left. \hspace*{3.5cm} + F \cos 3 s_0 \right), \end{array}\]
where 
\[ \hspace*{-0.1cm}\begin{array}{rcl}
C &\hspace*{-0.2cm}=\hspace*{-0.2cm}& 16\big( (d_2 \bar{c}_{\Theta} + d_1 \bar{s}_{\Theta}) (3-\cosh 2 \tau) - 4(d_1 \bar{c}_{\Theta} - d_2 \bar{s}_{\Theta}) \sinh \tau\big) \cosh^4 \tau \\[1ex]
&& + 12 \cosh^3 \tau \big(4 (d_2 \bar{c}_{\Theta} + d_1 \bar{s}_{\Theta}) \cosh \tau  \\[1ex] &&  \hspace*{2.2cm}- (d_1 \bar{c}_{\Theta} - d_2 \bar{s}_{\Theta}) (\cosh 2 \tau - 3) \coth \tau \big), \\[1.5ex]
D &\hspace*{-0.2cm}=\hspace*{-0.2cm}& 16\big( (d_1 \bar{c}_{\Theta} - d_2 \bar{s}_{\Theta}) (3-\cosh 2 \tau) + 4(d_2 \bar{c}_{\Theta} + d_1 \bar{s}_{\Theta}) \sinh \tau\big) \cosh^4 \tau \\[1ex]
&& + 12 \cosh^3 \tau \big(4 (d_1 \bar{c}_{\Theta} - d_2 \bar{s}_{\Theta}) \cosh \tau  \\[1ex] &&  \hspace*{2.2cm} + (d_2 \bar{c}_{\Theta} + d_1 \bar{s}_{\Theta}) (\cosh 2 \tau - 3) \coth \tau \big), \\[1.5ex]
E &\hspace*{-0.2cm}=\hspace*{-0.2cm}& 
( d_4 \tilde{c}_{\Theta}+ d_3 \tilde{s}_{\Theta}) ( 462 - 495 \cosh 2 \tau + 66 \cosh 4 \tau - \cosh 6 \tau) \\[1ex] && - 
4 (d_3 \tilde{c}_{\Theta} - d_4 \tilde{s}_{\Theta}) (198 \sinh \tau - 55 \sinh 3 \tau + 3 \sinh 5 \tau) \\[1ex]
&& + \mbox{$\frac{9}{4}$} \csch \tau \big ( (d_3 \tilde{c}_{\Theta} - d_4 \tilde{s}_{\Theta} ) ( 462 - 495 \cosh 2 \tau + 66 \cosh 4 \tau \\[1ex] && \hspace*{4.45cm}- \cosh 6 \tau) \\[1ex] && \hspace*{1.76cm} + 
   4 (d_4 \tilde{c}_{\Theta} + d_3 \tilde{s}_{\Theta}) (198 \sinh\tau - 55 \sinh 3 \tau + 
      3 \sinh 5 \tau) \big), \\[1.5ex]
F &\hspace*{-0.2cm}=\hspace*{-0.2cm}& 
( d_3 \tilde{c}_{\Theta} - d_4 \tilde{s}_{\Theta}) ( 462 - 495 \cosh 2 \tau + 66 \cosh 4 \tau - \cosh 6 \tau) \\[1ex] && + 
4 (d_4 \tilde{c}_{\Theta} + d_3 \tilde{s}_{\Theta}) (198 \sinh \tau - 55 \sinh 3 \tau + 3 \sinh 5 \tau)\\[1ex]
&& + \mbox{$\frac{9}{4}$} \csch \tau \big ( (d_4 \tilde{c}_{\Theta} + d_3 \tilde{s}_{\Theta}) ( -462 + 495 \cosh 2 \tau - 66 \cosh 4 \tau \\[1ex] && \hspace*{4.45cm} + \cosh 6 \tau) \\[1ex] && \hspace*{1.76cm} + 
   4 (d_3 \tilde{c}_{\Theta} - d_4 \tilde{s}_{\Theta}) (198 \sinh\tau - 55 \sinh 3 \tau + 
      3 \sinh 5 \tau) \big), 
\end{array} \]
and
\[
\begin{array}{l}
\bar{c}_\Theta = \cos\left(\mbox{$\Theta_0^3 \over 24 \varepsilon^3$} (9 \sinh \tau + \sinh 3 \tau)\right),  \\[1.4ex]
\bar{s}_\Theta = \sin\left(\mbox{$\Theta_0^3 \over 24 \varepsilon^3$} (9 \sinh \tau + \sinh 3 \tau)\right),\\[1.4ex]
 \tilde{c}_\Theta = \cos\left(\mbox{$\Theta_0^3 \over 8 \varepsilon^3$} (9 \sinh \tau + \sinh 3 \tau)\right),  \\[1.4ex]
\tilde{s}_\Theta = \sin\left(\mbox{$\Theta_0^3 \over 8 \varepsilon^3$} (9 \sinh \tau + \sinh 3 \tau)\right).\end{array}\]

We observe that the total derivative of the Duffing Hamiltonian has become
\[ \frac{d H_D}{d \tau} = \varepsilon^4 M_4 + \varepsilon^6 M_6 + O(\varepsilon^8). \]

At this point we need to compute the integrals of $M_4$ and $M_6$ for $\tau$ between $-\infty$ and $\infty$. After performing the change of variable $z = \sinh \tau$ (note that $z' = \cosh \tau > 0$, hence the change is well-defined), simplifying the resulting expressions and dropping the odd terms with respect to $z$ as their respective integrals are zero, we arrive at
\begin{equation}\label{Mel4}
 {\mathcal M}_4(s_0;\Theta_0,\varepsilon) = \int_{-\infty}^{\infty} \tilde{M}_4 \, d z = \pm \frac{2}{\Theta_0^6} {\mathcal F}_4(\Theta_0,\varepsilon) (c_2 \sin 2s_0 - c_3 \cos 2s_0), \end{equation}
where $\tilde{M}_4(z)=M_4(\tau(z))/\sqrt{z^2+1}$ and
\begin{equation} \label{FF4}\hspace*{-0.32cm}\begin{array}{rcl}
{\mathcal F}_4(\Theta_0,\varepsilon)  &\hspace*{-0.2cm}=\hspace*{-0.2cm}& \displaystyle \int_{-\infty}^{\infty} \frac{1}{(z^2+1)^6}\left(2 (7 z^4 -12 z^2 + 1) \cos (\mbox{$\Theta_0^3 \over 3 \varepsilon^3$} z (z^2+3)) \right.\\ &&  \displaystyle \left. \hspace*{1.45cm} + z ( 3z^4 - 26 z^2 + 11 ) \sin( \mbox{$\Theta_0^3\over 3 \varepsilon^3$} z (z^2+3))\right) dz. \end{array}\end{equation}
We have passed from $\tau_0$ to $s_0$ as the resulting expressions are much easier expressed in terms of them. From now the Melnikov function $\mathcal M$ will be written in terms of $s_0$.

In order to obtain the integral of $M_6$ with respect to $\tau$ we apply as before the change $z = \sinh \tau$, discard those terms with zero
integral and simplify, ending up with
\begin{equation}\label{Mel6}
\begin{array}{rcl}
{\mathcal M}_6(s_0;\Theta_0, \varepsilon) &\hspace*{-0.2cm}=\hspace*{-0.2cm}& \displaystyle \int_{-\infty}^{\infty} \tilde{M}_6 \, d z  \\[2ex]
\displaystyle &\hspace*{-0.2cm}=\hspace*{-0.2cm}&  \displaystyle\pm \frac{2}{\Theta_0^8} \big( {\mathcal F}_{6,1}(\Theta_0, \varepsilon) (d_2 \cos s_0 - d_1 \sin s_0)\\[1.5ex]&& \hspace*{1cm}\displaystyle + {\mathcal F}_{6,2}(\Theta_0, \varepsilon) (d_4 \cos 3 s_0 - d_3 \sin 3 s_0) \big), \end{array}\end{equation}
with $\tilde{M}_6(z)=M_6(\tau(z))/\sqrt{z^2+1}$ and
\begin{equation} \label{FF612}\begin{array}{rcl}
{\mathcal F}_{6,1}(\Theta_0, \varepsilon) &\hspace*{-0.3cm}=\hspace*{-0.3cm}& \displaystyle \int_{-\infty}^{\infty} \frac{1}{(z^2+1)^6}\left( (9z^2-1) \cos ( \mbox{$\Theta_0^3  \over 6 \varepsilon^3$} z (z^2+3)) \right.\\ &&  \displaystyle \left. \hspace*{2.23cm}+ 2z (2z^2-3) \sin ( \mbox{$\Theta_0^3  \over 6 \varepsilon^3$} z (z^2+3))\right) dz, \\[5ex] 
 {\mathcal F}_{6,2}(\Theta_0, \varepsilon) &\hspace*{-0.2cm}=\hspace*{-0.2cm}& \displaystyle \int_{-\infty}^{\infty} \frac{1}{(z^2+1)^8}
 \left( ( 27 z^6 - 125 z^4 + 69 z^2 -3 ) \times \right. \\[-0.2ex] && \left. \hspace*{2.7cm} \times \cos ( \mbox{$\Theta_0^3  \over 2 \varepsilon^3$} z (z^2+3)) \right.\\[0.5ex] &&  \displaystyle \left. \hspace*{-0.58cm} +
  2 z (2 z^6 - 39 z^4 + 60 z^2 - 11 ) \sin ( \mbox{$\Theta_0^3  \over 2 \varepsilon^3$} z (z^2+3))\right) dz.\end{array}\end{equation}

Higher order terms, say $M_8$, $M_{10}$, etc. are obtained in a similar way as $M_4$ and $M_6$. An important feature is that $\varepsilon^4 M_4$ is the dominant term, the second one is $\varepsilon^6 M_6$ and so on. However, the ordering of the functions ${\mathcal M}_{2k}$ is different, due to the contribution of $\varepsilon$ inside the arguments of the trigonometric functions. Indeed, this ordering depends now on the harmonics $\cos k s_0$, $\sin k s_0$, in such a way that the dominant terms are those corresponding to $\cos s_0$, $\sin s_0$, then those of $\cos 2 s_0$, $\sin 2s_0$ and so on
as we shall see in the Appendix. As a consequence, the terms in ${\mathcal M}_6$ containing $\cos s_0$, $\sin s_0$ represent the leading terms in the approximation in $\varepsilon$ of the Melnikov function of the restricted $N$-body problem (\ref{h1}) related to parabolic motions of the infinitesimal particle near infinity. 

It is straightforward to check that the relevant factor of the terms related to $\cos s_0$, $\sin s_0$ always arise through expressions of the form $\varepsilon^{2(2l+1)} {\mathcal M}_{2(2l+1)}$, $l \ge 2$ and it is
\[ d_2^{(l)} \cos s_0 - d_1^{(l)} \sin s_0 \quad \mbox{with} \,\,\, 
\begin{cases} d_1^{(l)} = \sum_{j=1}^{N-1} m_j a_{j1} (a_{j1}^2+a_{j2}^2)^l, \\[1ex] d_2^{(l)} = -\sum_{j=1}^{N-1} m_j a_{j2} (a_{j1}^2+a_{j2}^2)^l.\end{cases} \]  We state the following result.

\begin{theorem}\label{TeoM4}
There exists $\varepsilon_0 > 0$ such that for any $\varepsilon$ with $\varepsilon_0 \le \varepsilon \ll 1$ the stable and unstable manifolds of the periodic orbit $\gamma$ related to Hamiltonian (6) intersect transversally if one of the following situations is given: i) $d_1$ or $d_2$ do not vanish; ii) $d_1 = d_2 = 0$ and there exists $l \ge 2$ such that $d_1^{(l)}$ or $d_2^{(l)}$ do not vanish; iii) all the previous terms are zero and $c_2$ or $c_3$ do not vanish; iv) all the preceding terms are zero and there is a non-null constant term accompanying $\cos k s_0$ or $\sin k s_0$ for some $k \ge 2$.
\end{theorem}
\begin{proof}
In the Appendix, using an argument of Sanders \cite{Sanders} we prove that ${\mathcal R}$ can be maintained small enough in regions of the phase space that include parts of the stable and unstable manifolds of $\gamma$ such that transversality is satisfied. More precisely, we can control the size of
${\mathcal R}$ for all $|\tau| \ge K > 0$ and $K$ a constant independent of $\varepsilon$, keeping it smaller than the dominant terms of $\mathcal M$. Then we prove that the leading terms of the Melnikov function are those factorizing $\cos s_0$, $\sin s_0$, then those factorizing $\cos 2s_0$, $\sin 2s_0$ and so on. Therefore, the most influential term in the Melnikov function is the one appearing in $\varepsilon^6 {\mathcal M}_6$, that is $(d_2 \cos s_0 - d_1 \sin s_0) {\mathcal F}_{6, 1}(\Theta_0, \varepsilon)$. As ${\mathcal F}_{6, 1}$ does not vanish for $\Theta_0 \neq 0$ and $\varepsilon \ge \varepsilon_0$ small enough, we focus on the analysis of the possible zeros of the factor $f(s_0)=d_2 \cos s_0 - d_1 \sin s_0$.  Multiple roots of $f(s_0)=0$ occur only when $d_1 = d_2 = 0$. In this case we consider the next most important term of the Melnikov function, that is, $\varepsilon^{10} {\mathcal M}_{10}$. Its corresponding constant terms are $d_1^{(2)}$, $d_2^{(2)}$ and they always appear in the form $d_2^{(2)} \cos s_0 - d_1^{(2)} \sin s_0$, and similarly for $l > 2$. Thus, it is enough that there is a non-null coefficient $d_j^{(l)}$, $j = 1, 2$, $l \ge 2$, to establish the transversality condition. When all terms related to $\cos s_0$, $\sin s_0$ vanish, we take into account the next main terms, that are those of $\varepsilon^4 {\mathcal M}_4$. They appear in the form $c_3 \cos 2s_0 - c_2 \sin 2s_0$, thus it is enough that $c_2$ or $c_3$ do not vanish to get transversality. When $c_2 = c_3 = 0$ we need to consider the next terms related to $\cos 2s_0$, $\sin 2s_0$, and these terms appear in $\varepsilon^8 {\mathcal M}_8$. We continue the procedure until we identify a non-null constant term accompanying $\cos k s_0$ or $\sin k s_0$, concluding the transversality of the manifolds in that case.
\end{proof}

In Section \ref{sec5}, more specifically when we will deal with the circular restricted 3-body problem and with the polygonal restricted $N$-body problem, we shall see how higher-order terms of the Melnikov function are needed to establish the transversality of the stable and unstable manifolds of the parabolic periodic orbits.

%%%%%%%%%%%%%%
%%%%%%%%%%%%%%
\section{Applications}\label{sec5}
%%%%%%%%%%%%%%
%%%%%%%%%%%%%%

%%%%%%%%%%%%%%
%%%%%%%%%%%%%%
\subsection{Restricted circular 3-body problem}
%%%%%%%%%%%%%%
%%%%%%%%%%%%%%
We study the planar circular restricted 3-body problem. In this example the position and mass parameters can be chosen as
\[a_{11}=1-\mu,\quad
a_{12}=0,\quad
a_{21}=-\mu,\quad
a_{22}=0,\quad
m_1=\mu,\quad
m_2=1-\mu,
\]
where $0<\mu \le \frac{1}{2}$.

With these values, the perturbation parameters appearing in (\ref{ds}) read as:
\[ d_1 = 3 \mu (1-\mu) (1-2\mu),\quad d_2=0.\]
Then, according to Theorem \ref{TeoM4}, when $\mu < 1/2$ for $\varepsilon >0$ small enough the stable and unstable manifolds of the periodic orbit $\gamma$ related to Hamiltonian (\ref{pez}) intersect transversally.

For $\mu = 1/2$ we have to consider higher-order terms, starting with the harmonics $\cos s_0$, $\sin s_0$. In this case we note that the parameters $d_1^{(l)}$, $d_2^{(l)}$ vanish for all $l \ge 2$, then we calculate $c_2$ and $c_3$, arriving at \[ c_2 = \frac{3}{4}, \quad c_3 = 0. \]

Thus, Theorem \ref{TeoM4} applies and the stable and unstable manifolds of the periodic orbit $\gamma$ intersect transversally.

With the approach described above we have completed the analysis of the parabolic orbits for the circular restricted 3-body problem for any $\mu \in (0, 1/2]$. This problem is also treated in \cite{famous}.

%%%%%%%%%%%%%%
%%%%%%%%%%%%%%
\subsection{Equilateral restricted 4-body problem}
%%%%%%%%%%%%%%
%%%%%%%%%%%%%%
In this example the three massive particles form an equilateral triangle, therefore a central configuration, thus (\ref{cc}) is satisfied. The parameter values are:
\[
\begin{array}{ll}
a_{11}=\frac{1}{2}(1-m_1-2m_2), &\,\,
a_{12}=\frac{\sqrt{3}}{2}(1-m_1), \\[1ex]
a_{21}=\frac{1}{2}(2-m_1-2m_2), &\,\,
a_{22}=-\frac{\sqrt{3}}{2}m_1, \\[1ex]
a_{31}=-\frac{1}{2}(m_1+2m_2), &\,\,
a_{32}=-\frac{\sqrt{3}}{2}m_1,\\[1ex]
m_3=1-m_1-m_2,
\end{array}
\]
where $m_1,m_2>0$ and $m_1+m_2<1$.

Then, the coefficients $d_i$ of (\ref{ds}) take the following values:
\[
\begin{array}{l}
d_1=\mbox{$\frac{3}{2}$}(m_1 + 2 m_2 -1)(2m_1^2 + 2 m_2^2 + 2m_1 m_2-m_1-2m_2),\\[1ex]
d_2=-\mbox{$\frac{3\sqrt{3}}{2}$} m_1 \left( 2m_1^2 + 2m_2^2 + 2m_1 m_2 -3m_1-2m_2+1\right),\\[1ex]
\end{array}
\]
Therefore the transversal intersection of the manifolds is established as Theorem \ref{TeoM4} applies, except for the case $m_1=m_2=1/3$, which is the only case for which $d_1=d_2=0$ in the allowed region of the mass parameters.

When $m_1=m_2=1/3$ we check that the coefficients $d_1^{(l)}$, $d_2^{(l)}$ vanish for all $l \ge 2$.
Moreover, $c_2 = c_3 = 0$ and realize that all terms accompanying $\cos 2s_0$, $\sin 2s_0$ vanish as well. Thus, we consider the leading terms of $\cos 3s_0$, $\sin 3s_0$ and get 
\[ d_3 = 0, \quad d_4 = \frac{5}{3\sqrt{3}}. \]

Thence Theorem \ref{TeoM4} is applied, accomplishing the intersection of the manifolds of the parabolic orbits. 

The case $m_1=m_2=1/3$ is also included as a particular situation of the polygonal restricted $N$-body problem that will be addressed in Subsection \ref{polyN}.

%%%%%%%%%%%%%%
%%%%%%%%%%%%%%
\subsection{Restricted rhomboidal  5-body problem}
%%%%%%%%%%%%%%
%%%%%%%%%%%%%%
We consider the case where masses $m_1$ to $m_4$ are equal by pairs and form a convex polygon, a rhombus, see \cite{MarMe} and references therein. The parameters that define the problem are as follows:
\[
\begin{array}{llll}
a_{11}=-x, \quad&
a_{12}=0, &   
\hspace*{-1cm}
a_{21}=0, \quad&
\,\,a_{22}=y, \\[0.8ex]
a_{31}=x, \quad&
a_{32}=0, & 
\hspace*{-1cm}
a_{41}=0, \quad&
\,\,a_{42}=-y,\\[0.8ex]
m_1=m_3=\mu, &
m_2=m_4=\frac{1}{2}-\mu,
\end{array}\]
where $0 < \mu < 1/2$ and $x, y > 0$. To get a central configuration the parameters $\mu$, $x$, $y$ must be related. For this purpose we impose that Eqs. (\ref{cc}) are satisfied. It is convenient to introduce two parameters $a,b>0$ such that
\[\begin{array}{l}
x=\displaystyle \frac{a}{2\sqrt{a^2 + b^2}} \left(\frac{64 a^3 b^3 - (a^2 + b^2)^3}
{ 16 a^3 b^3 -(a^3+b^3)(a^2+b^2)^{3/2} }\right)^{1/3},\\[2ex]
 y=\displaystyle \frac{b}{2\sqrt{a^2 + b^2}} \left(\frac{64 a^3 b^3 - (a^2 + b^2)^3}
{ 16 a^3 b^3 -(a^3+b^3)(a^2+b^2)^{3/2} }\right)^{1/3},\end{array}\]
then a central configuration occurs provided $\mu$ is taken as
\[ \mu = \frac{a^3 \left(8 b^3-(a^2 + b^2)^{3/2} \right)}{2\left(16 a^3 b^3-(a^3 + b^3) (a^2 + b^2)^{3/2}\right)}. \]
To ensure that $\mu \in (0, 1/2)$ we must restrict $a$, $b$ so that $0 < b < \sqrt{3} a < 3 b$.

In terms of $x$, $y$, the perturbation parameters in (\ref{ds}) are identically zero. Moreover the coefficients $d_1^{(l)}$, $d_2^{(l)}$ also vanish for all $l \ge 2$, thus we need to obtain $c_2$, $c_3$. We get
\[
\begin{array}{l}
c_2=-3y^2+6\mu (x^2+y^2),\\[1.1ex]
c_3=0.
\end{array}
\]
When $c_2$ is non-zero Theorem \ref{TeoM4} applies and we get the transversality of the manifolds. For  $c_2=0$ we should go to higher orders in $\varepsilon$. The coefficient $c_2$ vanishes in three cases: (i) If $a = b$, then $\mu=1/4$, which corresponds to the square configuration, i.e. the polygonal restricted 4-body problem, that will be treated in the next subsection; (ii) when $a = 1.32018439... b$; and (iii) the reverse case to (ii), i.e. $a = 0.75746994... b$. The last two  values come as the only (two) real roots of the 14-degree homogeneous polynomial equation given by
\[\begin{array}{l}
a^{14} + 2 a^{13} b + 6 a^{12} b^2 + 10 a^{11} b^3 + 17 a^{10} b^4 + 
22 a^9 b^5 - 36 a^8 b^6 - 100 a^7 b^7 \\[1ex] 
- 36 a^6 b^8 + 22 a^5 b^9 + 
17 a^4 b^{10} + 10 a^3 b^{11} + 6 a^2 b^{12} + 2 a b^{13} + b^{14} = 0.
\end{array}\]
According to Theorem \ref{TeoM4}, the next terms that have to be checked for the three cases pointed out above are the coefficients of $\cos 2s_0$, $\sin 2s_0$ that appear in the function $\varepsilon^8 {\mathcal M}_8$. Specifically, we have calculated
$\mp 10\varepsilon^8 \Theta_0^{-10}{\mathcal F}_{8,1}(\Theta_0,\varepsilon) (c_2^{(2)} \sin 2 s_0 - c_3^{(2)} \cos 2 s_0)$ where ${\mathcal F}_{8,1}$ is a function of order $\varepsilon^{-5/2} \exp(-2\Theta_0^3/(3\varepsilon^3))$ when $\Theta_0>0$ and of order $\varepsilon^{-1} \exp(2\Theta_0^3/(3\varepsilon^3))$ when $\Theta_0<0$. Besides $c_2^{(2)}$, $c_3^{(2)}$ are coefficients depending on the masses and positions. We get $c_3^{(2)}=0$ for the three cases, $c_2^{(2)}=0.20447308...$ for case (ii) and $c_2^{(2)}=-0.20447308...$ for case (iii). Thus we can conclude the transversality of manifolds. However in case (i) one has $c_2^{(2)}=0$. This is the case $a = b$, that corresponds to a particular situation of the polygonal restricted $N$-body problem dealt with in Subsection \ref{polyN}.

%%%%%%%%%%%%%%
%%%%%%%%%%%%%%
\subsection{Collinear restricted $N$-body problem}
%%%%%%%%%%%%%%
%%%%%%%%%%%%%%
In the following we present two examples where the $N-1$ massive particles are collinear in a planar central configuration.  

Note that in all collinear problems $a_{k2}=0$ for $k=1,\ldots,N-1$. Then, the parameter $c_3$ of (\ref{c}) always vanishes. 

%%%%%%%%%%%%%%
%%%%%%%%%%%%%%
\subsubsection{Collinear restricted $8$-body problem}
%%%%%%%%%%%%%%
%%%%%%%%%%%%%%
Let us consider a collinear configuration of seven bodies with equal masses that are placed in a symmetric configuration with respect to the origin on the horizontal axis. In order to calculate the positions we impose that the conditions (\ref{cc}) are satisfied and so, the seven bodies form a central configuration. Then, by means of the determination of the resultant of two polynomials we find that a collinear configuration occurs if the parameters are chosen as follows:
\[
\begin{array}{l}
a_{11}=-1.17858061..., \quad
a_{21}=-0.73861375..., \\[1ex]
a_{31}=-0.35910513..., \quad
a_{41}=0,\\[1ex]
a_{51}=-a_{31},\,\, \qquad
a_{61}=-a_{21},\,\  \qquad
a_{71}=-a_{11}, \\[1ex]
a_{k2}=0, \quad m_k=\frac{1}{7},\quad {\rm for}\; k=1,\ldots,7.
\end{array}
\]

The parameters $d_1$, $d_2$ of (\ref{ds}) vanish. Furthermore the coefficients associated to the higher order terms of the harmonics $\cos s_0$, $\sin s_0$, i.e. $d_1^{(l)}$, $d_2^{(l)}$ are all zero for all $l\ge 2$. 

Then, the perturbation parameters appearing in (\ref{c}) are:
\[
c_2=1.76876487... ,\quad
c_3=0.
\]
Thus the conditions of Theorem \ref{TeoM4} are satisfied. Thence, for $\varepsilon$ small enough with $\varepsilon \ge \varepsilon_0 > 0$, the stable and unstable manifolds of the periodic orbit $\gamma$ related to Hamiltonian (\ref{pez}) intersect transversally.  

%%%%%%%%%%%%%%
%%%%%%%%%%%%%%
\subsubsection{Collinear restricted $11$-body problem}
%%%%%%%%%%%%%%
%%%%%%%%%%%%%%
Here we consider ten particles in collinear configuration placed at equidistant positions. Then, we calculate the masses so as to obtain a central configuration. For that, we impose that Eqs. (\ref{cc}) are satisfied and solve the resulting system of linear equations to get
\[
\begin{array}{l}
a_{11}=-1.44194062...,\qquad
a_{21}=-1.12150937...,\\[1ex]
a_{31}=-0.80107812...,\qquad
a_{41}=-0.48064687...,\\[1ex]
a_{51}=-0.16021562...,\\[1ex]
a_{61}=-a_{51},\,\,
a_{71}=-a_{41},\,\,
a_{81}=-a_{31},\,\,
a_{91}=-a_{21},\,\,
a_{101}=-a_{11},\\[1ex]
a_{k2}=0\quad {\rm for}\; k=1,\ldots,10,\\[1ex]
m_1=m_{10}=0.05585772..., \qquad
m_2=m_9=0.08684056..., \\[1ex]
m_3=m_8=0.10794726..., \qquad
m_4=m_7=0.121390422...,\\[1ex]
m_5=m_6=0.12796403... . 
\end{array}
\]

Note that, although here we put an approximation of the values correct to eight decimal digits, we have obtained them using integer arithmetic.

As in the previous case, the parameters $d_1$, $d_2$ of (\ref{ds}) become zero while the coefficients  $d_1^{(l)}$, $d_2^{(l)}$ also vanish. Then we need to obtain the leading terms of the harmonics $\cos 2s_0$, $\sin 2s_0$.
We calculate the coefficients~(\ref{c}) and get their values also exactly, an approximation of them accurate up to eight decimal places  being:
\[
c_2=1.95579995...,\quad
c_3=0.
\]
Then, as $c_2 \neq 0$, Theorem \ref{TeoM4} is satisfied, concluding the transversality condition of the manifolds of $\gamma$.

%%%%%%%%%%%%%%
%%%%%%%%%%%%%%
\subsection{Polygonal restricted $N$-body problem}
\label{polyN}
%%%%%%%%%%%%%%
%%%%%%%%%%%%%%
In this case particles 1 to $N-1$ form a regular $(N-1)$-gon determined by the following constants:
\[a_{k1}={\rm Re}({\rm e}^{2\pi i \frac{k-1}{N-1}}), \qquad
a_{k2}={\rm Im}({\rm e}^{2\pi i \frac{k-1}{N-1}}), \qquad 
m_k=\frac{1}{N-1},
\]
for $k=1$ to $N-1$ where $N\ge 4$. 

We want to write down Hamiltonian (\ref{hamm3}) for the specific cases of the polygonal restricted problem.
After some manipulations and simplifications that include an induction over the integer $N$, we notice that
the Hamiltonian function can be written in terms of symplectic polar coordinates in a rather compact way by
\begin{equation}\label{polig}
\begin{array}{rcl}
     H_\varepsilon(r, \theta, R, \Theta,t) &\hspace*{-0.2cm}=\hspace*{-0.2cm}&
	\displaystyle \varepsilon^3 \left( \frac{1}{2} \left( R^2 +\frac{\Theta^2}{r^2} \right) - \frac{1}{r} \right) - 
	\sum_{j=1}^{2N-3} \frac{\varepsilon^{j+3}}{r^{j/2+1}} U_j \\[1ex]
	&& \displaystyle - 
	 \frac{\varepsilon^{2N+1}}{r^N} \Big(V_{N-1}
	+ W_{N-1} \cos (N-1) (t-\theta) \Big) \\[2ex] && +O(\varepsilon^{2N+2}),
	\end{array} \end{equation}
where
\[ \begin{array}{rcl}
U_j &\hspace*{-0.2cm}=\hspace*{-0.2cm}& \displaystyle \frac{\left(1+(-1)^j+2 \cos(j\pi/2)\right) (\Gamma(j/4+1/2))^2}{4\pi (\Gamma(j/4+1))^2}, \\[2.5ex]
V_{N-1} &\hspace*{-0.2cm}=\hspace*{-0.2cm}& \displaystyle\frac{(1+(-1)^{N-1})(\Gamma(N/2))^2}{2\pi (\Gamma(N/2+1/2))^2}, \\[2.5ex]
W_{N-1} &\hspace*{-0.2cm}=\hspace*{-0.2cm}& \displaystyle \frac{2 \Gamma(N-1/2)}{\sqrt{\pi} \Gamma(N)},
\end{array} \]	
and $\Gamma$ stands for the gamma function.

An important feature of (\ref{polig}) is that terms of order higher than $\varepsilon^{2N+1}$ can depend on $\cos(N-1)(t-\theta)$ but they are of smaller influence than the one of order $\varepsilon^{2N+1}$.

To obtain the Melnikov function we emphasize the convenience of developing $H_\varepsilon$ to order $2N+1$ because the first appearance of $\theta$
occurs at this order and the previous orders would yield zero.

At this point we apply the same steps as in the previous sections, arriving at the total derivative 
\[ \begin{array}{rcl}
\displaystyle \frac{d H_D}{d \tau} &\hspace*{-0.2cm}=\hspace*{-0.2cm}& \displaystyle -\mbox{$\frac{1}{4}$}\Theta_0^2 \sinh 2 \tau \sum_{j=1}^{2N-3}
  \varepsilon^j U_j \frac{2^{ j/2+1} (j+2)}{ (|\Theta_0| \cosh \tau)^{j+4}} \\
  && \displaystyle
  -\varepsilon^{2N-2} \frac{2^{N}}{\Theta^{2 N}_0 \cosh^{2N+1} \tau}
  \left(N V_{N-1} \sinh \tau  \right. \\[1ex] && \displaystyle \left. \hspace*{4.22cm} + W_{N-1} \left(N \sinh \tau \cos q(s_0,\tau) \right. \right. \\[1ex] && \displaystyle \left. \left. \hspace*{5.7cm} \mp (N-1) \sin q(s_0, \tau)\right)\right) \\ && + O(\varepsilon^{2N-1}),
        \end{array}\]
with
        \[  
        q(s_0,\tau) = (N-1) \big(s_0 \mp 4 \arctan(\tanh(\tau/2)) \pm
       \mbox{$ \frac{\Theta_0^3}{24\varepsilon^3} $} (9 \sinh \tau + \sinh 3 \tau)\big).  \]
       
Now, we observe that by the parity of the derivative with respect to $\tau$, the only term with no zero integral
in the derivative is that factorized by $W_{N-1}$, thus we introduce 
\[ \begin{array}{rcl} \varepsilon^{2N-2} M_{2N-2} &\hspace{-0.2cm}=\hspace{-0.2cm}& \displaystyle -\varepsilon^{2N-2} \frac{2^N W_{N-1}}{\Theta^{2 N}_0 \cosh^{2N+1} \tau} \left( N \sinh \tau\cos q(s_0,\tau) \right.\\ && \left.\hspace*{4.25cm} \mp (N-1) \sin q(s_0,\tau) \right).\end{array} \]
We perform the change $z = \sinh \tau$ and define the Melnikov function as
\begin{equation}\label{Melnipol} \begin{array}{rcl}
{\mathcal M}_{2N-2}(s_0; \Theta_0, \varepsilon) &\hspace*{-0.3cm}=\hspace*{-0.3cm}& \displaystyle -\frac{2^N W_{N-1}}{\Theta_0^{2N}} 
\int_{-\infty}^{\infty} \frac{1}{(z^2+1)^{N+1}} \left(N z \cos \tilde{q}(s_0,z) \right. \\ &&\left.  \hspace*{3cm} \mp (N-1) \sin \tilde{q}(s_0,z) \right) d z, \end{array}\end{equation}
with \[
\tilde{q}(s_0,z) = (N-1) \big(s_0 \pm \mbox{$\frac{\Theta_0^3}{6\varepsilon^3}$} z(z^2+3) \mp 4 \arctan(\tanh(\mbox{${1\over 2}$}{\rm arcsinh}z)) \big).  \]

In order to illustrate how the theory of this paper applies, we particularize the calculations for two specific cases, namely $N = 7, 8$.

When $N = 7$ we get
\[    {\mathcal M}_{12}(s_0; \Theta_0, \varepsilon) = \pm  \frac{231}{4 \Theta_0^{14}} {\mathcal F}_{12}(\Theta_0, \varepsilon) \sin 6 s_0 \] where
\[\begin{array}{rcl}
 {\mathcal F}_{12}(\Theta_0, \varepsilon) &\hspace*{-0.5cm}=\hspace*{-0.5cm}& \displaystyle 
 \int_{-\infty}^{\infty} \frac{1}{(z^2+1)^{14}}\left( p_1(z) \cos (\mbox{$\frac{\Theta^3_0}{\varepsilon^3}$}  z (z^2+3))\right.\\ &&  \displaystyle \left. \hspace*{2.82cm} + p_2(z) \sin (\mbox{$\frac{\Theta^3_0}{\varepsilon^3}$} z (z^2+3))\right) dz, \end{array}\]
 with \[ \hspace*{-0.2cm}\begin{array}{l}
 p_1(z) = 2 (45 z^{12}- 968 z^{10} + 4257 z^8- 5544 z^6+ 2255 z^4 - 240 z^2 + 3), \\[1ex]
 p_2(z) = z ( 7 z^{12} - 534 z^{10} + 4785 z^8 - 11220 z^6 + 8217 z^4  - 1782 z^2 + 79 ). \end{array} \]
 
 For $N = 8$ the corresponding Melnikov function reads as
\[{\mathcal M}_{14}(s_0; \Theta_0, \varepsilon) = \pm \frac{429}{\Theta_0^{16}} {\mathcal F}_{14}(\Theta_0, \varepsilon) \sin 7 s_0 \] where
\[\begin{array}{rcl}
 {\mathcal F}_{14}(\Theta_0,\varepsilon) &\hspace*{-0.5cm}=\hspace*{-0.5cm}& \displaystyle 
 \int_{-\infty}^{\infty} \frac{-1}{(z^2+1)^{16}}\left( p_3(z) \cos (\mbox{$\frac{7\Theta^3_0}{6\varepsilon^3}$} z (z^2+3))\right.\\ &&  \displaystyle \left. \hspace*{2.85cm} + p_4(z) \sin (\mbox{$\frac{7\Theta^3_0}{6\varepsilon^3}$} z (z^2+3))\right) dz, \end{array}\]
 with \[ \begin{array}{l}
 p_3(z) = 119 z^{14} - 3549 z^{12}  + 23023 z^{10} - 48477 z^8 + 37037 z^6 - 9919 z^4 \\
 \hspace*{1.4cm} + 749 z^2 - 7, \\[1ex]
 p_4(z) = 2 z (
  4 z^{14} - 413 z^{12} + 5278 z^{10} - 19019 z^8 + 24024 z^6 - 11011 z^4  \\
 \hspace*{2cm} + 1638 z^2 - 53). \end{array} \]

From the expressions of both Melnikov functions it is clearly deduced that the equations ${\mathcal M}_{12}=0$ and ${\mathcal M}_{14}=0$ have simple roots provided, respectively, ${\mathcal F}_{12}$ and ${\mathcal F}_{14}$ do not vanish. However ${\mathcal F}_{12}$, ${\mathcal F}_{14}$ are of the same type as the functions ${\mathcal F}_4$, ${\mathcal F}_{6,1}$, ${\mathcal F}_{6,2}$ of Section \ref{sec4} analyzed in the Appendix. Then they have a global maximum and a global minimum, their graphs cut the horizontal axis at a few points and tend asymptotically to zero as long as $\Theta_0$ tends to $\pm \infty$, see also Figure \ref{fig_F12-14.pdf}. 

\begin{figure}
\centering
\includegraphics[scale=0.66]{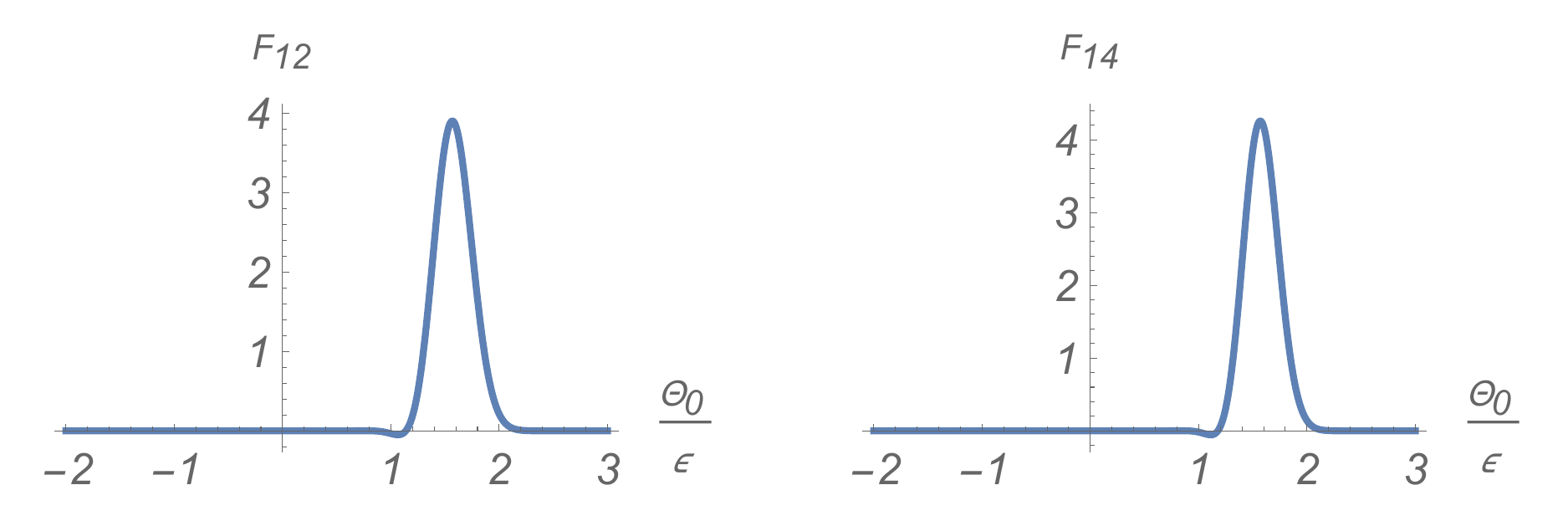}
\caption{On the left: graph of ${\mathcal F}_{12}$; on the right, graph of ${\mathcal F}_{14}$.}
\label{fig_F12-14.pdf}
\end{figure}

The asymptotic estimates of the Appendix hold provided $\varepsilon$ is small enough so that ${\mathcal F}_{12}$ and ${\mathcal F}_{14}$ do not vanish. Thus, for the polygonal restricted $7$- and  $8$-body problems, the stable and unstable manifolds of the parabolic orbits $\gamma$ intersect transversally.

Regarding the case $N = 4$, corresponding to the equilateral restricted 4-body problem of Subsection 6.2 when $m_1 = m_2 = m_3 = 1/3$, the Melnikov function obtained from (\ref{Melnipol}) is a particular case of the function ${\mathcal M}_6$ in (\ref{Mel6}), thus we can achieve the transversality from the analysis of this subsection. When $N = 5$, the square configuration studied in Subsection 6.3 is analyzed using the Melnikov function (\ref{Melnipol}) which is of order eight in $\varepsilon$. Proceeding similarly to what we did for $N=7,8$ we conclude that ${\mathcal M}_8$ has simple roots provided, thus extending the analysis performed in Subsection 6.3 when $a = b$.

Examining (\ref{Melnipol}) carefully, it is not difficult to infer that the integrals appearing in the Melnikov function ${\mathcal M}_{2N-2}$ are of the same type as ${\mathcal F}_4$, ${\mathcal F}_{6,1}$ and so on. We also take into account that the smallest harmonic appearing in the Melnikov function is $\sin (N-1) s_0$, thus we end up with an expression like
\[  {\mathcal M}_{2N-2}(s_0;\Theta_0, \varepsilon) = \pm  \frac{K}{\Theta_0^{2N}} {\mathcal F}_{2N-2}(\Theta_0, \varepsilon) \sin (N-1) s_0, \]
with $K$ a non-null constant.
Applying the estimates provided in the Appendix and taking $\varepsilon$ small enough, we conclude that $\varepsilon^{2N-2} {\mathcal M}_{2N-2}$ behaves like  $\varepsilon^{-N-1/2}\exp(-(N-1) \Theta_0^3/(3 \varepsilon^3))(1+O(\varepsilon))$ when $\Theta_0>0$ and $\varepsilon^{-N+1}\exp((N-1) \Theta_0^3/(3 \varepsilon^3))(1+O(\varepsilon))$ for negative $\Theta_0$. Thus, 
we can apply Theorem \ref{TeoM4}, achieving the transversality of the manifolds of $\gamma$
in the polygonal restricted $N$-body problem for all $N\ge 4$.

%%%%%%%%%%%%%%
%%%%%%%%%%%%%%

\section{Appendix: Qualitative study of functions ${\mathcal F}_4(\Theta_0,\varepsilon)$, ${\mathcal F}_{6,1}(\Theta_0,\varepsilon)$ and ${\mathcal F}_{6,2}(\Theta_0,\varepsilon)$}

The function $\mathcal{F}_4$ has been defined in (\ref{FF4}) and its graph is given in Figure \ref{configura}. Considered as a function in $z$, the integrand is smooth in $\mathbb{R}$ and bounded by above in the intervals in $(-\infty,-1] \cup [1,\infty)$ by the improperly integrable function $80/|z|^7$. Calling $\tilde{\Theta}_0=\Theta_0/\varepsilon$, the comparison test for improper integrals implies that ${\mathcal F}_4(\tilde{\Theta}_0)$ is a well defined smooth function for all $\tilde{\Theta}_0 \in \mathbb{R}$. In addition it has two zeroes, namely $\tilde{\Theta}_0^{'}=0$ and $\tilde{\Theta}_0^\ast = 0.61078210...$, and ${\mathcal F}_4(\tilde{\Theta}_0)<0$ for $\tilde{\Theta}_0 < \tilde{\Theta}_0^\ast$ (excepting at $\tilde{\Theta}_0^{'}$), while ${\mathcal F}_4(\tilde{\Theta}_0)>0$ for $\tilde{\Theta}_0 > \tilde{\Theta}_0^\ast$. The function $\mathcal{F}_4$ takes a global maximum and a global minimum values. Furthermore
\[ \lim_{\tilde{\Theta}_0 \to \pm \infty}  {\mathcal F}_4(\tilde{\Theta}_0) = 0.\]

\begin{figure}[h]
\centering
\includegraphics[scale=.45]{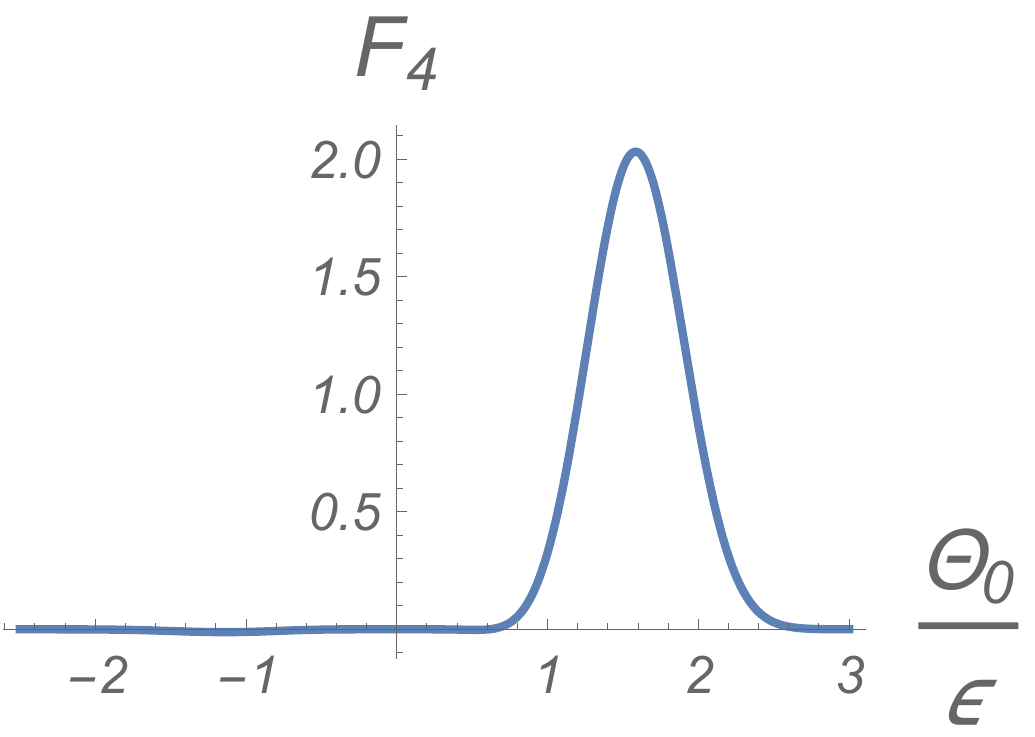}
\caption{The graph of the function ${\mathcal F}_4(\Theta_0,\varepsilon)$.}
\label{configura}
\end{figure}

The functions $\mathcal{F}_{6,1}$, $\mathcal{F}_{6,2}$ were introduced in (\ref{FF612}) and present an analogous behavior to the function $\mathcal{F}_{4}$, as it can be seen in Figure \ref{fig612}. The corresponding improper integrals are absolutely convergent. Specifically, $\mathcal{F}_{6,1}=0$ has its unique root at $\tilde{\Theta}_0=0$ whereas the roots of $\mathcal{F}_{6,2}=0$ occur at $\tilde{\Theta}_0=0$, $\tilde{\Theta}_0=0.15745028...$, $\tilde{\Theta}_0=0.87685728...$. Besides, $\mathcal{F}_{6,1} > 0$ when $\tilde{\Theta}_0<0$, $\mathcal{F}_{6,1} < 0$ when $\tilde{\Theta}_0 > 0$ while $\mathcal{F}_{6,2} > 0$ when $\tilde{\Theta}_0 < 0$ and $0.15745028... < \tilde{\Theta}_0 < 0.87685728...$, $\mathcal{F}_{6,2} < 0$ when $0 < \tilde{\Theta}_0 < 0.15745028...$ and $\tilde{\Theta}_0 > 0.87685728...$. As ${\mathcal F}_4$, the functions $\mathcal{F}_{6,1}$, $\mathcal{F}_{6,2}$ take a global maximum as well as a global minimum. Finally, 
\[ \lim_{\tilde{\Theta}_0\to \pm \infty}  {\mathcal F}_{6,1}(\tilde{\Theta}_0) = \lim_{\tilde{\Theta}_0\to \pm \infty}  {\mathcal F}_{6,2}(\tilde{\Theta}_0) = 0.\]

\begin{figure}
\centering\includegraphics[scale=.66]{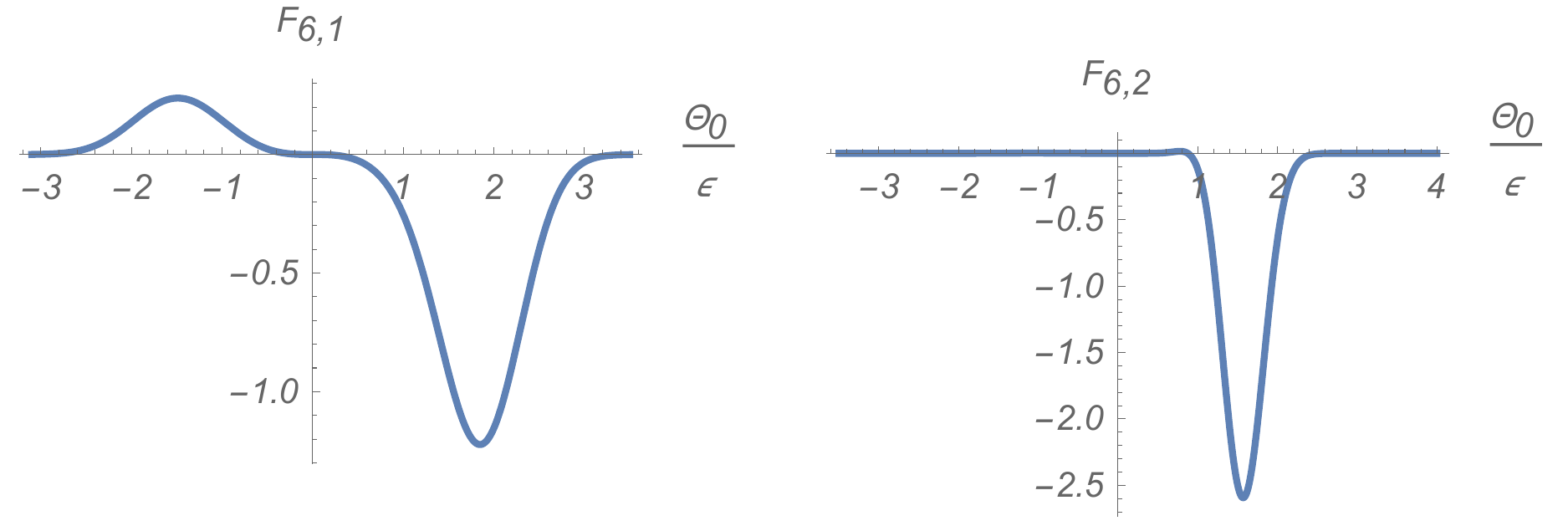} \caption{The graphs of the functions ${\mathcal F}_{6,1}(\Theta_0,\varepsilon)$
and ${\mathcal F}_{6,2}(\Theta_0,\varepsilon)$.}
\label{fig612}
\end{figure}

We remark that similar integrals have been analyzed and can be found in \cite{Regina,RegSim}. In fact, due to the highly oscillatory character of the integrals, an asymptotic analysis of the functions $\mathcal{F}_4$, $\mathcal{F}_{6,1}$, $\mathcal{F}_{6,2}$ and other related functions appearing in
the Melnikov functions obtained in Section \ref{sec4} has to be done.

Following \cite{RegSim} we introduce the improper integrals
\[ I_k(\delta) = \int_0^{\infty} \frac{\cos (\delta (z+z^3/3))}{(1+z^2)^k} dz, \quad 
   J_k(\delta) = \int_0^{\infty} \frac{z \sin (\delta (z+z^3/3))}{(1+z^2)^k} dz.\]
In \cite{Regina} it is proved that $J_k$ can be written in terms of $I_k$ through
\[ J_{k+2}(\delta) = \frac{\delta}{2(k+1)} I_k(\delta), \]
whereas for $\delta >0$ big enough, the following estimates hold:
\begin{equation} \label{esti} \begin{array}{rcl}
I_{2n-1}(\delta) &\hspace*{-0.2cm}=\hspace*{-0.2cm}& \displaystyle \exp(-2\delta/3) \left( \frac{\pi}{2^{n+1}(2n-2)!!} \delta^{n-1} + O(\delta^{n-3/2})\right), \\[1.8ex]
I_{2n}(\delta) &\hspace*{-0.2cm}=\hspace*{-0.2cm}& \displaystyle \exp(-2\delta/3) \left( \frac{\sqrt{\pi}}{2^{n+1}(2n-1)!!} \delta^{n-1/2} + O(\delta^{n-1}) \right). \end{array} \end{equation}

The functions $\mathcal{F}_4$, $\mathcal{F}_{6,1}$, $\mathcal{F}_{6,2}$ as well as the rest of the functions appearing in ${\mathcal M}_{2k}$ with $k > 3$ can be cast in terms of $I_k$ and $J_k$, after performing a partial fraction decomposition to each of the rational parts of them.

We start with $\Theta_0 > 0$ and take $\varepsilon$ small enough with $\varepsilon \ge \varepsilon_0 > 0$ to avoid the possible zeroes of the functions ${\mathcal F}_4$, ${\mathcal F}_{6,1}$, etc. Concerning ${\mathcal M}_4$ we apply the estimates (\ref{esti}) to $\mathcal{F}_4$, and after arranging the function  conveniently, we conclude that 
\[ \varepsilon^4 {\mathcal M}_4 = \mbox{$\frac{4 \sqrt{\pi}}{3}$} \varepsilon^{-7/2} \Theta_0^{3/2} e^{-\frac{2\Theta_0^3}{3\varepsilon^3}}(c_2 \sin 2 s_0 - c_3 \cos 2 s_0) (1+ O(\varepsilon)). \]

For ${\mathcal M}_6$ we get
\[ \begin{array}{rcl}
\varepsilon^6 {\mathcal M}_6 &\hspace*{-0.2cm}=\hspace*{-0.2cm}& -\mbox{$\frac{\sqrt{\pi}}{12\sqrt{2}}$} \varepsilon^{-3/2} \Theta_0^{-1/2} e^{-\frac{\Theta_0^3}{3\varepsilon^3}}(d_2 \cos s_0 - d_1 \sin s_0) (1 + O(\varepsilon) )\\
 &&-\mbox{$\frac{9\sqrt{3\pi}}{5\sqrt{2}}$} \varepsilon^{-9/2} \Theta_0^{5/2} e^{-\frac{\Theta_0^3}{\varepsilon^3}}(d_4 \cos 3s_0 - d_3 \sin 3s_0) (1+ O(\varepsilon)), \end{array} \]
 and similar expressions are obtained for ${\mathcal M}_8$, ${\mathcal M}_{10}$, and so on.
 
From the above calculations it is clear that for $\varepsilon$ small enough the most important terms are those related to $\exp(-\Theta_0^3/(3\varepsilon^3))$, then those related to $\exp(-2\Theta_0^3/(3\varepsilon^3))$, next those with $\exp(-\Theta_0^3/\varepsilon^3)$ and so on. Furthermore, we observe that the terms with asymptotic estimates having the factor $\exp(-\Theta_0^3/(3\varepsilon^3))$ correspond to the harmonics $\cos s_0$, $\sin s_0$, and in general, the asymptotic expressions with $\exp(-k \Theta_0^3/(3\varepsilon^3))$ are related to the harmonics $\cos k s_0$, $\sin k s_0$. In addition to this, for $k \ge 2$ the terms of $\cos k s_0$, $\sin k s_0$ are of order $\varepsilon^{-k-3/2} \exp(-k \Theta_0^3/(3\varepsilon^3))(1+ O(\varepsilon))$ while for $k=1$ the terms of $\cos s_0$, $\sin s_0$ have the estimate $\varepsilon^{-3/2} \exp(-\Theta_0^3/(3\varepsilon^3))(1+ O(\varepsilon))$. This in turn implies that the leading terms in the Melnikov function are the ones depending on the coefficients $d_1$, $d_2$, the next ones those with estimate $\varepsilon^{-1/2} \exp(-\Theta_0^3/(3\varepsilon^3))$ to which follows the rest of terms with harmonics $\cos s_0$, $\sin s_0$. Next we consider the main terms factorized by $\cos 2s_0$, $\sin 2s_0$, they are the ones depending on $c_2$, $c_3$. We continue with the higher order terms and so on. The Melnikov function 
becomes the formal Fourier series
\[{\mathcal M}(s_0; \Theta_0, \varepsilon) = \sum_{k=1}^\infty \alpha_k(\varepsilon) \cos k s_0 + \beta_k(\varepsilon) \sin k s_0 \]
with
\[ \begin{array}{l}
\alpha_1(\varepsilon) = \varepsilon^{-3/2} e^{-\frac{\Theta_0^3}{3\varepsilon^3}} (A_1 + O(\varepsilon)), \quad  \,\,\beta_1(\varepsilon) = \varepsilon^{-3/2} e^{-\frac{\Theta_0^3}{3\varepsilon^3}} (B_1 + O(\varepsilon)), \\[0.6ex]
\alpha_k(\varepsilon) = \varepsilon^{-k-3/2} e^{-\frac{k \Theta_0^3}{3\varepsilon^3}} (A_k + O(\varepsilon)), \,\beta_k(\varepsilon) = \varepsilon^{-k-3/2} e^{-\frac{k \Theta_0^3}{3\varepsilon^3}} (B_k + O(\varepsilon)), \end{array} \] for $k \ge 2$ and
\[ \begin{array}{ll}
 A_1 = -\frac{\sqrt{\pi}}{12\sqrt{2}} \Theta_0^{-1/2} d_2, &  
 B_1 = \frac{\sqrt{\pi}}{12\sqrt{2}} \Theta_0^{-1/2} d_1, \\[1.5ex]
 A_2 = \frac{4\sqrt{\pi}}{3} \Theta_0^{3/2} c_3, &  
 B_2 = -\frac{4\sqrt{\pi}}{3} \Theta_0^{3/2} c_2, \\[1.5ex]
&\hspace*{-4.2cm} A_k, B_k, \,\, k \ge 3 \,\,\,\, \mbox{constants independent of} \,\,\, \varepsilon.
 \end{array}\]

Regarding the estimate of ${\mathcal R}$ we use the ideas of Sanders \cite{Sanders} for the case
of exponentially small estimates. Given a vector field of the form $\dot{x} = f_0(x)+\varepsilon f_1(x,t,\varepsilon)$ with $x\in D\subset \mathbb{R}^2$ and $\varepsilon$ a small parameter, he defines the Melnikov integral $\Delta_\varepsilon(t,x)=\varepsilon^{-1} f_0(x^{s,u}_0(t))\wedge(x_\varepsilon^u(t)-x_\varepsilon^s(t))$ where $\wedge$ is the wedge product in $\mathbb{R}^2$, $x^{s,u}_0$ refers to the parametrization of the stable and unstable manifolds of the unperturbed system and $x_\varepsilon^u$, $x_\varepsilon^s$ denote solutions on 
the unstable and stable manifolds, respectively. After some assumptions on the smoothness of the vector field, and a lemma regarding the relationships between the unperturbed and perturbed manifolds,  Sanders arrives at an expression of the form
\[ \Delta_\varepsilon(t,x) = \Delta_0(x) + O(\varepsilon (1+e^{-\mu |t|})^2 \min \{1, e^{-\mu |t|}\}),\]
where $\Delta_0$ stands for the usual Melnikov function and $\mu$ is the Lipschitz constant associated to $f_0$. Systems of the type $\dot{x} =\varepsilon f(x,t,\varepsilon)$, after applying averaging and rescaling time by $\tau = \varepsilon t$, are transformed into $y'=dy/d\tau = f_0(y) + \varepsilon f_1(y, \tau/\varepsilon, \varepsilon)$, thus admitting the estimate $O(\varepsilon (1+\exp({-\mu |\tau|/\varepsilon}))^2 \min\{1, \exp({-\mu |\tau|/\varepsilon}) \})=O(\varepsilon \exp({-\mu |\tau|/\varepsilon}))$ for $\tau \neq 0$.

We adopt Sanders' point of view in our setting as follows. First, we notice that the hypotheses 
on the Hamiltonian and the existence of the manifolds are fulfilled. Then we realize that $t$ and $\tau$ are related through the change of time, assuming that $x(\tau) = \sqrt{2}/\Theta_0 \sech \tau+O(\varepsilon^3)$ we get $t = \Theta_0^3\tau/(2\varepsilon^3) + O(1)$. The Lipschitz constant of $H_D$ along the homoclinic $\xi(\tau)$ is calculated, obtaining $\mu = \sqrt{2}/\Theta_0$. Finally, after adjusting the factor $\varepsilon^4$ in the whole expression of $\Delta_\varepsilon$ so that we identify $\Delta_0$ with $\mathcal M$, we get an upper bound on $\mathcal R$ as
$O(\exp({-\Theta_0^2 |\tau|/(\sqrt{2} \varepsilon^3)}))$. To control the size of $\mathcal R$ we compare the estimate with the dominant term of $\mathcal M$. When $d_1$ or $d_2$ are not zero, due to the 
presence of the factor $\varepsilon^{-3/2}$ in $\alpha_1$, $\beta_1$, it is enough that $\exp({-\Theta_0^3/(3\varepsilon^3)}) > \exp({-\Theta_0^2 |\tau|/(\sqrt{2} \varepsilon^3)})$ from where it is deduced that $|\tau| \ge \frac{\sqrt{2}}{3} \Theta_0$, which is true in big portions of the manifolds as $\Theta_0$ is of moderate size. Next, the transversality condition is verified
 in the part of phase space where this restriction on $\tau$ holds, but it implies that transversality is satisfied for every $\tau$ as this property is preserved through diffeomorphisms. When $d_1 = d_2 = 0$ one compares $\exp({-k\Theta_0^3/(3\varepsilon^3)})$ with $\exp({-\Theta_0^2 |\tau|/(\sqrt{2} \varepsilon^3)})$, starting with $k= 2$.

When $\Theta_0 < 0$ we notice that to apply the estimates given above we should consider the integrals $I_k$, $J_k$ with $\delta < 0$. Then we notice that $I_k(\delta) = I_k(-\delta)$, $J_k(\delta) = -J_k(-\delta)$ and $J_{k+2}(\delta) = \delta/(2(k+1)) I_k(-\delta)$ and the estimates for $I_k$ given in (\ref{esti}) apply replacing $\delta$ by $-\delta$ in the expressions. Proceeding similarly to the case $\Theta_0 > 0$ we arrive at
\[ \begin{array}{rcl}
\varepsilon^4 {\mathcal M}_4 &\hspace*{-0.2cm}=\hspace*{-0.2cm}& \mbox{$\frac{5\pi}{8}$} \varepsilon^{-2} e^{\frac{2\Theta_0^3}{3\varepsilon^3}}(c_2 \sin 2s_0 - c_3 \cos 2s_0 ) (1 + O(\varepsilon) ), \\[1ex]
\varepsilon^6 {\mathcal M}_6 &\hspace*{-0.2cm}=\hspace*{-0.2cm}& -\mbox{$\frac{5\pi}{128}$} \Theta_0^{-2} e^{\frac{\Theta_0^3}{3\varepsilon^3}}(d_2 \cos s_0 - d_1 \sin s_0) (1 + O(\varepsilon) ) \\
 &&+ \mbox{$\frac{63\pi}{64}$} \varepsilon^{-3} \Theta_0 e^{\frac{\Theta_0^3}{\varepsilon^3}}(d_4 \cos 3s_0 - d_3 \sin 3s_0) (1+ O(\varepsilon)). \end{array} \]

Reasoning as in the positive case we realize that the main term in the Melnikov function is that of $\varepsilon^6 {\mathcal M}_6$ having coefficients $d_1$, $d_2$, then the rest of terms related to $\cos s_0$, $\sin s_0$, next the term associated to $\cos 2s_0$, $\sin 2s_0$ beginning with those whose coefficients are $c_2$, $c_3$, and so on. Moreover, when $k \ge 2$ the terms of $\cos k s_0$, $\sin k s_0$ behave like $\varepsilon^{-k} \exp(k \Theta_0^3/(3\varepsilon^3))(1+ O(\varepsilon))$, thus we obtain analogous results to the case $\Theta$ positive although with estimates of different order in $\varepsilon$.

The estimate analysis of $\mathcal R$ is alike the procedure for $\Theta_0>0$.

\section*{Acknowledgements}
We appreciate the comments of Prof. M. Guardia, P. Mart{\'\i}n and T.M. Seara who pointed out a crucial error in the previous version of the paper and helped substantially in its improvement.
The authors have received partial support from Project Grant Red de Cuerpos Acad\'emicos de Ecuaciones Diferenciales, Sistemas Din\'amicos y Estabilizaci\'on. PROMEP 2011-SEP, Mexico and from Projects 2014--59433--C2--1--P of the Ministry of Economy and Competitiveness of Spain and 2017--88137--C2--1--P of the Ministry of Economy, Industry and Competitiveness of Spain. The facilities provided by Cinvestav -- IPN (Mexico) are also acknowledged.


\begin{thebibliography}{99}
\bibitem{alekseev} Alekseev, V.M.: Quasi-random oscillations and qualitative problems of celestial mechanics. In: Ninth Summer School for Mathematics / 2nd revised edition. Kiev, Izdatel'stvo Naukova Dumka, 1976, pp. 212-341, 359, 360. In Russian.

\bibitem{MarMe} Alvarez-Ram\'{\i}rez, M., Medina, M.: The rhomboidal 4-body problem revisited, {\em Qual. Theory Dyn. Syst.} {\bf 14} (2015) 189-207.

\bibitem{arnold64} Arnold, V.I.: Instability of dynamical systems with several degrees of freedom, {\em Soviet Math. Dok.} {\bf 5} (1964) 581-585.

\bibitem{chazy} Chazy, M.J.: Sur l'allure du mouvement dans le probl\`eme des trois corps quand le temps cro{\^\i}t ind\'efiniment, {\em Ann. Sci. \'Ec. Norm. Sup\'er.} $3^e$ s\'erie {\bf 39} (1922) 29-130.

\bibitem{chinos} Cheng, X., She, Z.: Study on chaotic behavior of the restricted four-body problem with an equilateral triangle configuration, {\em Internat. J. Bifur. Chaos Appl. Sci. Engrg.} {\bf 27} (2017) 1750026, 12 pp.
	
\bibitem{GMSS2017} Guardia, M., Mart\'{\i}n, P., Sabbagh, L., Seara, T.M.: Oscillatory orbits in the restricted elliptic planar three body problem, {\em Discrete Contin. Dyn. Syst.} {\bf 37} (2017) 229-256.

\bibitem{famous} Guardia, M., Mart{\'\i}n, P., Seara, T.M.: Oscillatory motions for the restricted planar circular three body problem, {\em Invent. Math.} {\bf 203} (2016) 417-492.

\bibitem{gh} Guckenheimer, J., Holmes, P.: {\em Nonlinear Oscillations, Dynamical Systems, and Bifurcations of Vector Fields}, Springer, New York, 2002.
	
\bibitem{HolMar1} Holmes, P.J., Marsden, J.E.: Melnikov's method and Arnold diffusion for perturbations of integrable Hamiltonian systems, {\em J. Math. Phys.} {\bf 23} (1982) 669-675.

\bibitem{HolMar2} Holmes, P.J., Marsden, J.E.: Horseshoes and Arnold diffusion for Hamiltonian systems on Lie groups, {\em Indiana Univ. Math. J.} {\bf 32} (1983) 273-309.

\bibitem{LlibreSimo} Llibre, J., Sim\'o, C.: Oscillatory solutions in the planar restricted three-body problem, {\em Math. Ann.} {\bf 248} (1980) 153-184.

\bibitem{Regina} Mart{\'\i}nez, R., Pinyol, C.: Parabolic orbits in the elliptic restricted three body problem, {\em J. Differential Equations}  {\bf 111} (1994) 299-339.

\bibitem{RegSim} Mart{\'\i}nez, R., Sim\'o, C.: Invariant manifolds at infinity of the RTBP and the boundaries of bounded motion, {\em Regul. Chaotic Dyn.}  {\bf 19} (2014) 745-765.

\bibitem{mcgehee} McGehee, R.: A stable manifold theorem for degenerate fixed points with applications to celestial mechanics, {\em J. Differential Equations} {\bf 14} (1973) 70-88. 

\bibitem{melnikov63} Mel'nikov, V.K.: On the stability of a center for time-periodic perturbations, {\em Trans. Moscow Math. Soc.} {\bf 12} (1963) 3-52.

\bibitem{meyer81} Meyer, K.R.: Periodic orbits near infinity in the restricted $N$-body problem, {\em Celestial Mech.} {\bf 23} (1981) 69-81.

\bibitem{meyer99} Meyer, K.R.: {\em Periodic Solutions of the $N$-Body Problem}, Lecture Notes in Mathematics, Springer-Verlag, New York, 1999.

\bibitem{meyer09} Meyer, K.R., Offin, D.C.: {\em Introduction to Hamiltonian Dynamical Systems and the $N$-Body Problem}, Third edition, Springer International Publishing AG, 2017.

\bibitem{moser} Moser, J.: {\em Stable and Random Motions in Dynamical Systems. With special emphasis on celestial mechanics}, Hermann Weyl Lectures. The Institute for Advanced Study, {\em Ann. Math. Stud.} {\bf 77}, Princeton University Press, Princeton, NJ, 1973.  

\bibitem{Robin} Robinson, C.: {\em Dynamical Systems: Stability, Symbolic Dynamics, and Chaos}, Second edition, CRC Press, Inc. (1999).

\bibitem{Sanders} Sanders, J.A.: Melnikov's method and averaging, {\em Celestial Mech.} {\bf 28} (1982) 171-181.

\bibitem{sitnikov} Sitnikov, K.A.: The existence of oscillatory motions in the three-body problem. In: Doklady Akademii Nauk SSSR {\bf 133} 303-306 (1960) (English Translation in {\em Soviet Phys. Dokl.} {\bf 5} 647-650 (1960)).

\bibitem {Xia92} Xia, Z.: Melnikov method and transversal homoclinic points in the restricted three-body problem, {\em J. Differential Equations} {\bf 96} (1992) 170-184.

\end{thebibliography}
\end{document}